\documentclass[a4paper,11pt]{amsart}

\usepackage{a4wide}

\usepackage{graphicx}
\usepackage{pictex}
\usepackage{epstopdf}
\usepackage{etex}

\usepackage{hyperref}
\usepackage{color}
\usepackage{amsthm}

\usepackage{latexsym}
\usepackage{autograph,epic,latexsym,bezier,amsbsy,color,enumerate,amsfonts,amsmath,amscd,amssymb}
\usepackage[latin1]{inputenc}

 \newlength\headseptemp

\numberwithin{equation}{section}    
\theoremstyle{plain}
\newtheorem{Theorem}{Theorem}[section]
\newtheorem{Proposition}[Theorem]{Proposition}
\newtheorem{Corollary}[Theorem]{Corollary}
\newtheorem{Lemma}[Theorem]{Lemma}
\theoremstyle{definition}

\newtheorem{Definition}[Theorem]{Definition}

\theoremstyle{remark}
\newtheorem{Remark}[Theorem]{Remark}

\newcommand{\RR}{\mathbb{R}}
\newcommand{\CC}{\mathbb{C}}
\newcommand{\NN}{\mathbb{N}}

\newcommand{\ZZ}{\mathbb{Z}}

\newcommand{\word}{\mathcal{W}}
\newcommand{\graph}{\mbox{gr}}
\newcommand{\Graph}{\mbox{Gr}}

\begin{document}

\title[Spectra of Schreier graphs of Grigorchuk's group]{Spectra of Schreier graphs of Grigorchuk's group and Schroedinger
operators with aperiodic order}

\author{Rostislav Grigorchuk}
\address{Mathematics Department, Texas A\&M University, College Station, TX 77843-3368, USA}
\email{grigorch@math.tamu.edu}

\author{Daniel Lenz}
\address{Mathematisches Institut \\Friedrich Schiller
Universit{\"a}t Jena \\07743 Jena, Germany }
\email{daniel.lenz@uni-jena.de}

\author{Tatiana Nagnibeda}
\address{Section de Math\'{e}matiques, Universit\'{e} de Gen\`{e}ve, 2-4, Rue du
Li\`{e}vre, Case Postale 64 1211 Gen\`{e}ve 4, Suisse}
\email{Tatiana.Smirnova-Nagnibeda@unige.ch}

\date{\today}

\begin{abstract}
We study spectral properties of the Laplacians on Schreier graphs
arising from  Grigorchuk's group acting on the boundary of the
infinite binary tree. We establish a connection between the action
of $G$ on its space of Schreier graphs and a subshift associated to
a non-primitive substitution and relate the Laplacians on the
Schreier graphs to discrete Schroedinger operators with aperiodic
order. We use this relation to prove that the spectrum of the
anisotropic Laplacians is a Cantor set of Lebesgue measure zero. We
also use it to show absence of eigenvalues both almost-surely and
for certain specific graphs. The methods developed here apply to a
large class of examples.
\end{abstract}

\maketitle



\unitlength=0,4mm \textwidth = 16.00cm \textheight = 22.00cm
\oddsidemargin= 0.12in \evensidemargin = 0.12in
\setlength{\parindent}{8pt} \setlength{\parskip}{5pt plus 2pt minus
1pt} \setloopdiam{10}\setprofcurve{7}

\section*{Introduction}
In this article we relate two previously unconnected areas. These
are   Schreier graphs of  self-similar groups and Schroedinger
operators associated to aperiodic order. This allows us to solve a
problem, which was open for some fifteen years, viz  the spectral
type problem for the Laplacians on Schreier graphs of the first
group of intermediate growth, for all possible weights attached to
the generators.

The first group of intermediate growth, introduced by the first
author in \cite{Gri80, Gri84}, is generally known as Grigorchuk's
group $G$ and this is how we will refer to it.\footnote{in spite of
the first author's reluctance}
The group $G$ can be viewed as a group of automorphisms of the full
infinite binary tree $\mathcal T$. The action by automorphisms on
the tree $\mathcal T$ extends by continuity to an action by
homeomorphisms on the boundary $\partial \mathcal T$ of the tree.
These actions give rise to Schreier graphs (see Section
\ref{Section-Schreier} for detailed  definition): for every $n\in
\NN$, a finite graph $\varGamma_n$ arises from the (transitive)
action on the $n$-th level of the tree; and  for every infinite ray
$\xi \in\partial \mathcal T$ an infinite graph $\varGamma_\xi$
arises from the orbit of $\xi$ under  the action of the group on the
boundary $\partial \mathcal T$. As Schreier graphs, these graphs
have edges labeled by the generators of the group in such a way
that each vertex has exactly one edge of each label incident with
it. The group $G$ naturally comes with a set of $4$ generators $\{a,b,c,d\}$ all of which are involutions, so in this case we can disregard the
orientation.

Thinking of  the labels $a,b,c,d$ as encoding some weights
 $t,u,v,w\in \RR$ on the edges of the graphs,  we obtain the  Laplacians
$M_n (t,u,v,w)$ and $M_\xi (t,u,v,w)$ acting on the $l^2$-space on
the vertex set of the corresponding graph. It is the spectral theory
of these Laplacians  that is meant by spectral theory of the
Schreier graphs of $G$. In the case when $t,u,v,w >0$ with $t + u +
v  + w = 1$ these operators are the Markov operators of the random
walk on these graphs with transition probabilities $t,u,v,w$.

Spectral theory of these Laplacians  was first studied in the
article \cite{BG}, and methods developed there were successfully
applied to various examples of self-similar groups
\cite{GS1,GS2,GN,GZ}. However, only the \lq\lq isotropic\rq\rq\ case
 with the values $u = v = w$ that imply periodicity of the operators $M_\xi (t,u,v,w)$, could be treated so far. The spectral problem in
the case of arbitrary weights remained open.

Here, we solve this case  by a new and completely different
approach. Our approach relies on the construction of a subshift
associated to a certain substitution intimately related to the
algebraic structure of the group. This substitution $$\kappa:
a\mapsto a c a, b\mapsto d, c\mapsto b, d \mapsto c$$ first appeared
in connection with Grigorchuk's group in the work of Lysenok
\cite{Lys}, where it was used to get a recursive presentation of $G$
by generators and relations:
$$G = \langle a,b,c,d \mid 1 = a^2 = b^2 = c^2 = d^2 = bcd =  \kappa^k ((ad)^4)
 =\kappa^k ((adacac)^4), k = 0,1,2,....\rangle .$$
It is remarkable that the substitution $\kappa$ serves not only to
define $G$ algebraically, but also, as will be shown here, to
describe its spectral properties and to determine $G$ in terms of
topological dynamics as a subgroup of the topological full group of
a minimal Cantor system.

For us it will be convenient to recode the substitution $\kappa$  on
a new alphabet arising by replacing the letters b,c,d by y,x,z. We
will call the corresponding substitution $\tau$. Our approach then
relies on constructing and studying the subshift
$(\varOmega_\tau,T)$ associated to $\tau$. This is done in Section
\ref{The-substitution}. There,  we also discuss the occurrences of
third powers in the subshift, which will be used later in our
discussion of absence of eigenvalues. Moreover, we also give a
fairly complete study of some  additional basic  properties of the
subshift, which are not necessary for the treatment of spectral
theory. In particular, we determine its index (which is four), and
show how the fixed point of the substitution can be generated by an
automaton.

The crucial main step of the paper, carried out in Section
\ref{Connecting}, is to relate the subshift $(\varOmega_\tau,T)$ to
the dynamical system $(X,G)$ (defined in Section
\ref{Section-Schreier}) arising from the action of the group on the
space of its Schreier graphs. More specifically, our main result,
Theorem \ref{Main-connection} shows that $\varOmega_\tau$ admits an
action of $G$ and
 there
exists a continuous surjective map
$$\psi : \varOmega_\tau \longrightarrow X$$
intertwining the respective actions of $G$. Hence,  $(X,G)$ is a
factor of $(\varOmega_\tau,G)$). Theorem  \ref{Main-connection} is
of interest in itself: for example, it implies that  $G$ is a
subgroup of the topological full group of the subshift defined by
the substitution, see Remark \ref{Connection-Bon}. In our context,
it allows us to show that there exists a family of Schroedinger
operators $(H_\omega)_{\omega\in\varOmega_\tau}$ associated to the
subshift such that  the operator $M_{\psi(\omega)} (t,u,v,w)$ is
unitarily equivalent to $H_\omega$ (Proposition
\ref{unitary-equivalence}). In this way the spectral problem for the
$M_x (t,u,v,w)$, $x\in X$,  is translated into a spectral problem
for the $H_\omega$, $\omega \in\varOmega_\tau$.

The operators $(H_\omega)$ belong to a class known as Schroedinger
operators with aperiodic order. Such operators have  been studied
intensely over the last thirty  years, see e.g. the surveys
\cite{Dam,DEG}, both for its remarkable properties and for its
connection to the so-called quasicrystals (see e.g.
\cite{BGr,BM,KLS,Moody} for further information on this topic).
Combining the  combinatorial information on  $(\varOmega_\tau,T)$
obtained in  Section \ref{The-substitution} with the general theory
of Schroedinger operators with aperiodic order we are able to
determine the spectral theory of the family $(H_\omega)$, and thus
also the spectral theory of the $M_x (t,u,v,w)$, $x\in X$, by the
established unitary equivalence. Our main results show that their
spectrum is a Cantor set of Lebesgue measure zero (Theorem
\ref{Main-Laplacian-Cantor}) and that the absence of eigenvalues
holds almost surely as well as for some specific points (Theorem
\ref{Main-Laplacian-Absence-eigenvalues}). This implies in
particular purely singular continuous spectrum almost surely and for
some specific points.

\smallskip

We choose, in this paper, to concentrate on one example,
Grigorchuk's group $G$, which is the most prominent example in the
theory of groups acting on rooted trees. (In fact, the theory grew
out of this example!) Our results extend readily to various families
of groups acting by automorphisms of regular rooted trees, provided
that the Schreier graphs of the induced action on the boundary of
the tree are linear. The most famous such case is given by  the
uncountable family $\{G_\omega\}_{\omega\in\{0,1,2\}^\Bbb N}$
constructed in \cite{Gri80}, in which the group $G$ corresponds to
the sequence $\omega = (012)^\infty$; but there are also other
families generalizing Grigorchuk's group $G$, see \cite{Sun, BDN,
Nek1}. To such a family of groups indexed by sequences over a
certain finite alphabet, a family of subshifts can be associated,
similarly to what is done here in the case of the group $G$, though
the subshift may not be defined by a substitution if the group in
question is not self-similar. Part of our results extend to all such
groups, others require some additional conditions depending on the
particular infinite sequence. An interesting new aspect that appears
in this more general study is how the corresponding dynamical and
spectral properties vary over the family. The details are to follow
in a forthcoming paper.

\textbf{Acknowledgments.} R. G. was partially supported by the NSF grant DMS-1207669 and by ERC AG COMPASP.
 The authors acknowledge support of the Swiss
National Science Foundation. Part of this research was carried out
while D. L. and R. G. were visiting the Department of mathematics of
the University of Geneva. The hospitality of the department is
gratefully acknowledged. The authors also thank Yaroslav Vorobets
for allowing them to use his figures $3$ and $4$. Finally, the
authors would like to thank the anonymous referee for a  careful
reading of the manuscript resulting in various helpful suggestions.

\section{Grigorchuk's group $G$,  its Schreier
graphs and the associated Laplacians}\label{Grigorchuk}
In this section we introduce the main object of our interest:
Grigorchuk's group $G$ and the Laplacians on the associated Schreier graphs.


\subsection{Grigorchuk's group G}
\label{Group-G}

Let us denote by $\mathcal T$ the \textit{rooted binary  tree}. The
vertex set of $\mathcal T$ is given by $\{0,1\}^\ast$, i.e. the set
of all words over the alphabet $\{0,1\}$. The root of $\mathcal T$
is the empty word. There is an edge between
 $v$ and  $w$ whenever $w = v  k$  or $v = w k$ holds
for some $k\in\{0, 1\}$. The words $w \in\{0,1\}^n$ constitute the
\textit{$n$-the level} of the tree. (In the tree, they are at
combinatorial distance exactly $n$ from the root.)
The set $\{0,1\}^\NN$   of one-sided  infinite words  can be
identified with the  boundary  $\partial \mathcal T $ of $\mathcal T
$ consisting of infinite geodesic rays in $\mathcal T$ emanating
from the root (i.e. infinite paths starting at the root all of whose
edges are pairwise different). When equipped with the product
topology  $\{0,1\}^\NN$ is a  compact space.

The central object of our study is  Grigorchuk's group $G$. It is
generated by four automorphisms $a,b,c,d$ of the rooted binary tree
defined on the vertices via
\begin{itemize}
\item[] $  a(0 w) = 1 w$,  $ a (1 w) = 0 w;$
\item[] $  b(0w) = 0 a (w)$, $  b (1 w) = 1 c(w);$
\item[] $  c(0w) = 0 a (w)$,  $c (1 w) = 1 d(w);$
\item[] $  d(0w) = 0 w$,  $  d(1w) = 1 b(w),$
\end{itemize}
where the vertex $w$ is an arbitrary  word over  $\{0,1\}$.
Grigorchuk's group is a self-similar group and the automorphisms can
also be expressed in the self-similar form as
$$
a=\epsilon(id,id), \qquad b=e(a,c), \qquad  c=e(a,d), \qquad
d=e(id,b),
$$
where $e$ and $\epsilon$ are, respectively, the trivial and the
non-trivial permutations in the group $Sym(2)$ and $id$ is the
identity acting on the tree. We refer the interested reader to
\cite{Nek,GNS} for more details and  information about self-similar
groups.

Observe that all the generators are involutions and that
$\{1,b,c,d\}$ commute and constitute a group isomorphic to   the
Klein group $\ZZ_2\times \ZZ_2$. Let as also mention that there are
many more relations and the group is not finitely presented.  For
our subsequent discussion it will be important that $G$ acts
\textit{transitively} on each level, i.e. for arbitrary words $w,u$
over $\{0,1\}$ with the same length there exists a $g\in G$ with $g
u = w$.

\subsection{The Schreier graphs of $G$ and the dynamical system
$(X,G)$}\label{Section-Schreier}
The action of  the group $G$   on the vertices of the rooted binary
tree and on its boundary  induces on these sets the structure of
\textit{Schreier graphs}, with respect to the generating set
$\{a,b,c,d\}\subset G$ (see Figure \ref{Schreier-finite}). These are
labeled graphs and we will freely use standard notation on such
graphs as discussed e.g. in \cite{Vor2}. 
Specifically, for $z\in V(\mathcal T) \cup
\partial \mathcal T$ the Schreier graph $\Gamma_z$ has as its
vertices the \textit{orbit} $G z$  of $z$ under the action of $G$
and there is an edge with label $s\in\{a,b,c,d\}$ and between the
vertices $x$ and $y$ if and only if $s x = y$ holds. Note that  the
edges can indeed be seen as undirected   as all generators
$\{a,b,c,d\}$ are involutions.  The graphs $\varGamma_w$ and
$\varGamma_v$ coincide (as non-rooted graphs) whenever $v$ and $w$
are in the same orbit  of the action of $G$.  For the first three
levels of the tree the resulting graphs are shown in Figure
\ref{Schreier-finite}.
\begin{figure}[h!]
\begin{center}
\hspace*{-.5cm}
\includegraphics[scale=0.75]{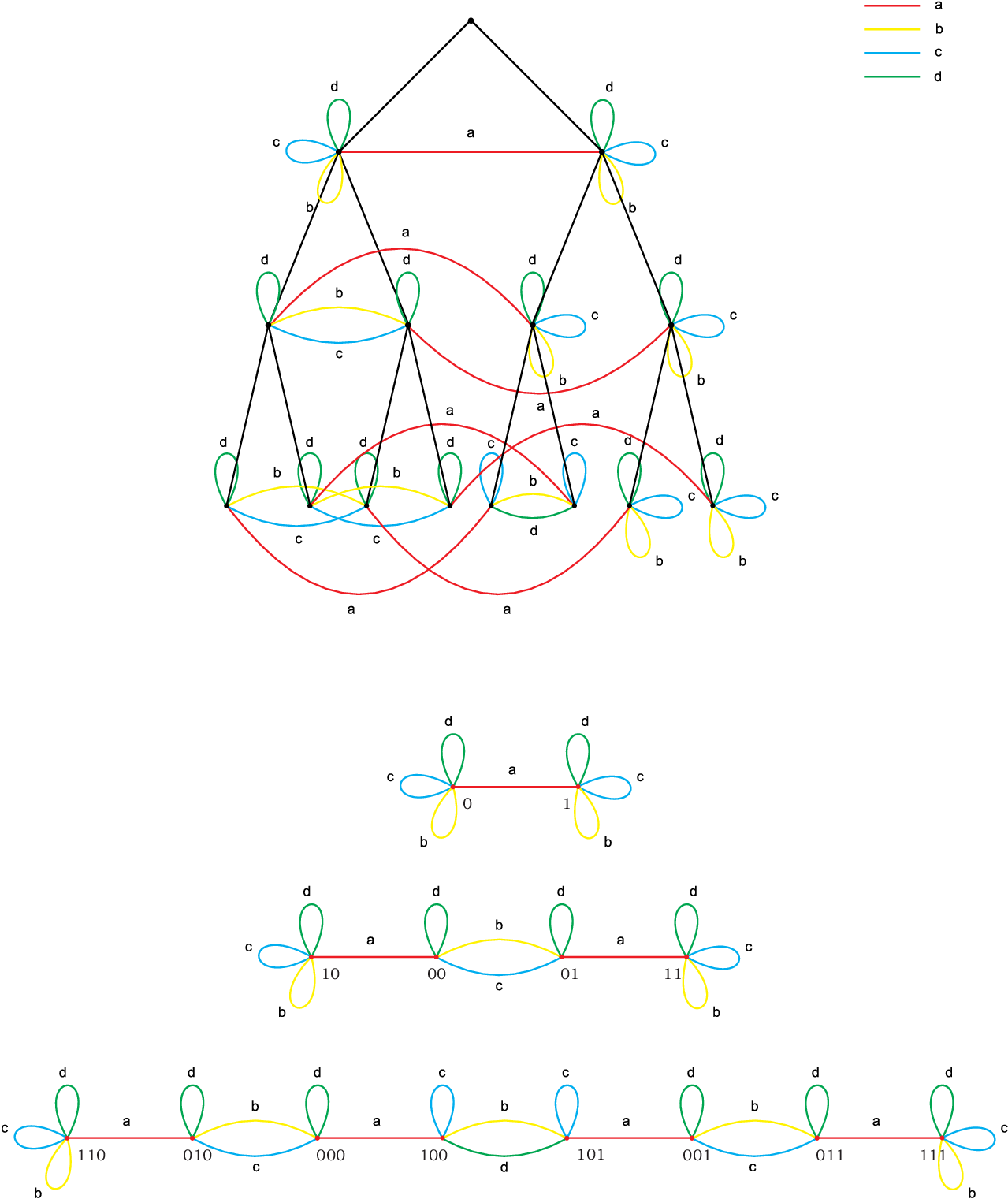}
\caption{The finite Schreier graphs of the first, second and third
level.} \label{Schreier-finite}
\end{center}
\end{figure}
As $G$ acts transitively on each level of the tree, for each $n\in
\NN$  the graph
$$\varGamma_n:= \varGamma_{1^n}$$
coincides with $\varGamma_w$ for all $w\in V(\mathcal T)$ with
$|w|=n$.

From  \cite{BG} we infer the following description of the Schreier
graphs: The Schreier graph $\Gamma_n$ has  $2^n$ vertices and  a
linear shape; it has $2^{n-1}$ simple edges, all labeled by $a$, and
$2^{n-1}-1$ cycles of length 2 whose edges are labeled by $b,c,d$.
It is regular of degree 4. The Schreier graphs corresponding to the
orbits of the action on the boundary are infinite and will be
referred to as \textit{orbital Schreier graphs}. The orbital
Schreier graphs have either two ends or one end. The graph
$\Gamma_{1^\infty}$ corresponding to the orbit of the rightmost
infinite ray in $T$ is one-ended, see Figure \ref{One-ended}, (and
so are then clearly all graphs in  the same orbit).
\begin{figure}[ht]
\begin{center}
\hspace*{-.5cm}
\includegraphics[scale=0.75]{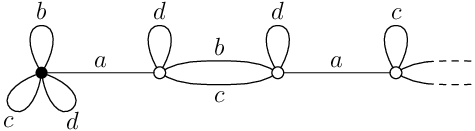} \caption{The one-ended graph $\varGamma_{1^\infty}$.} \label{One-ended}
\end{center}
\end{figure}
All  the other orbital Schreier graphs $\Gamma_\xi, \xi\notin G\cdot
1^\infty$,  are two-ended. They are all isomorphic as unlabeled
graphs.

We will view the Schreier graphs as rooted and be interested in them
up to isomorphism classes only. Thus, we introduce the map
$$\mathcal{F} :V(\mathcal{T})\cup
\partial \mathcal{T}  \longrightarrow \mathcal G_*
(\{a,b,c,d\}),\;\mathcal F (v):=[ (\varGamma_v, v)],$$ where
$\mathcal G_* (\{a,b,c,d\})$ is the space of isomorphism classes of
rooted graphs with labels in $\{a,b,c,d\}$ and  $(\varGamma_v,v)$
stands for the graph $\varGamma_v$ with root $v$ and $[\cdot]$
denotes the isomorphism class.
As is well-known (see e.g. Section 3.1 in \cite{BG} or Section 2.2
of \cite{DDMN} and compare \cite{GZ99} as well) finite Schreier
graphs converge to infinite orbital Schreier graphs, as follows.
Here and in the sequel  we use the usual  topology of local
convergence  on the space $ \mathcal G_* (\{a,b,c,c\})$  of
isomorphism classes of rooted connected graphs with  labels in
$\{a,b,c,d\}$, as discussed e.g. in \cite{Vor2}.

\begin{Lemma}\label{F-continuous}
 For every $\xi = \xi_1 \xi_2\ldots \xi_n\ldots  \in \{0,1\}^\NN$, the sequence
$(\mathcal F (\xi_1\ldots \xi_n))$ converges to  $\mathcal F (\xi)$.
\end{Lemma}


In \cite{Vor2}, Vorobets studied the closure $\overline{\mathcal
F(\partial \mathcal T)}$  in the space $ \mathcal G_*
(\{a,b,c,c\})$.  He showed that
the one-ended graphs are exactly the isolated points of this closure
$\overline{\mathcal F(\partial \mathcal T)}$, and that the other
points in $\overline{\mathcal F(\partial \mathcal T)}$ are two ended
graphs. This suggests to consider the compact subset of the space $\mathcal G_*(\{a,b,c,c\})$
$$X := \overline{\mathcal F(\partial \mathcal T)} \setminus
\{\text{isolated points}\}$$
which is called \textit{the space of Schreier graphs} of $G$.
 Then, the  group $G$ acts on  $X$ by
changing the root of the graph and this action is minimal (i.e. each
orbit is dense)  and uniquely ergodic, i.e., admits  a unique
invariant probability measure which will be denoted as  $\nu$. The dynamical system $(X,G)$ will be the  focus of attention in later sections.

In \cite{Vor2}, the following precise description of $X$ is given. The space $X$
is the disjoint union $$X = X_1 \sqcup X_2$$ of two sets $X_1$ and
$X_2$. Here, $X_2$ consists of all isomorphism classes of two-ended
rooted Schreier graphs $\{(\Gamma_\xi,\xi)  : \xi\in\partial
\mathcal T \setminus G\cdot 1^\infty\}$. The set $X_1$ consists of
three  countable families of isomorphism classes  of two-ended
graphs. These families are obtained by gluing two copies of the
one-ended graph $\Gamma_\xi, \xi\in G\cdot 1^\infty$,  at the root
in three possible ways corresponding to choosing a pair $(b,c)$,
$(b,d)$ or $(c,d)$, and then choosing an arbitrary vertex of the
arising graph as the root. One of these three possibilities is shown
in Figure \ref{Connecting-two-copies}. There, the chosen pair is
$(c,d)$ and to avoid confusion with other edges with the same
labels, the labels at the gluing point are denoted with a prime (and
the root is chosen arbitrarily).  These new graphs are again
Schreier graphs of $G$.

\begin{figure}[ht]
\begin{center}
\hspace*{-.5cm}
\includegraphics[scale=.75]{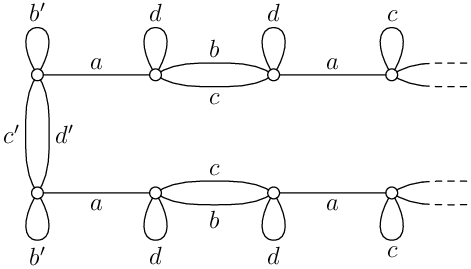}
\caption{Connecting two copies of $\Gamma_{1^\infty}$}
\label{Connecting-two-copies}
\end{center}
\end{figure}

Then, $(\partial \mathcal T,G)$ is a factor of $(X,G)$,  viz there
exists a surjective continuous map
$$\phi : X\longrightarrow
\partial \mathcal T$$ intertwining the respective
actions of $G$. Indeed, the decomposition of $X$ into isomorphism
classes of the $(\varGamma_\xi,\xi)$ and the three families
mentioned above gives rise to this map and  $\phi : X_2
\longrightarrow \mathcal{T}\setminus G\cdot 1^\infty$ is one-to-one
and onto with inverse given by $\mathcal{F}$ and $\phi : X_1
\longrightarrow G \cdot 1^\infty$ is three-to-one and onto.


\subsection{Laplacians associated to the Schreier graphs of
$G$}\label{Operators} In this section we introduce the operators
whose spectral theory is our main concern in this paper.


%



Let  $\varGamma$ be a graph with vertex set $V$ and  edges labeled
by $a,b,c,d$. Replacing the labels by $t,u,v,w\in\RR$ we obtain a
weighted graph and this graph gives rise to the  Laplacian
$M_\varGamma (t,u,v,w)$.  Specifically, the Laplacian $M_\varGamma
(t,u,v,w)$ is the  selfadjoint operator
$$M_\varGamma (t,u,v,w): \ell^2 (V)\longrightarrow \ell^2 (V)$$
acting via
$$(M_\varGamma (t,u,v,w) \varphi) (p) = \sum_{q\stackrel{e}{\sim} p} w(e)
f(q).$$ Here, the sum is over all edges $e$ emanating from $p\in V$
and the weight $w (e)$ of the edge $e$ is given by $t,u,v,w$
depending on whether its label is $a,b,c,d$ respectively. In the
case $\varGamma = \varGamma_\xi$ with  $\xi \in \partial \mathcal T$
we  write $M_\xi (t,u,v, w)$ instead of $M_{\varGamma_\xi}
(t,u,v,w)$. For $x\in X$ (which is an isomorphism class of graphs)
we choose a representative $\widehat{x}$ and then  write - with a
slight abuse of language - $M_x (t,u,v,w)$ to denote the operator
$M_{\widehat{x}} (t,u,v,w)$.

\section{The substitution $\tau$,  its subshift
$(\varOmega_\tau,T)$ and the associated finite words
$\mbox{Sub}_\tau$}\label{The-substitution}
 In this section we
study the two-sided subshift induced by a particular substitution.
The one-sided subshift induced by this substitution had already been
studied by Vorobets \cite{Vor1}, and some of our  results can be seen as
 two-sided counterparts to his.  His
investigation relies on a connection to Toeplitz sequences. Here, we
develop a new approach based on what we call the $n$-decomposition
and the $n$-partition of the elements of the subshift.  This is
close in spirit to the partition-based approach to Sturmian
dynamical systems which was developed in \cite{DL0}  and then applied in the spectral theory (see
the survey \cite{Dam}) and in some combinatorial questions
\cite{DL5,DL6}. For general background on subshifts we
refer to e.g. the textbook \cite{LM}.

Whenever $\mathcal{A}$ is a finite set, referred to as
\textit{alphabet},  we will consider the set $\mathcal{A}^\ast$ of
finite words (including the empty word)
as well as the set $\mathcal{A}^\ZZ$
of two-sided infinite words over the alphabet $\mathcal{A}$ and the
set $\mathcal{A}^\NN$
of one-sided infinite words over $\mathcal{A}$.
If $v,w$ are finite words and $\omega\in \mathcal{A}^\ZZ$ satisfies
$ \omega_1 \ldots \omega_{|v|} = v\:\; \mbox{and}\:\:
\omega_{-|w|+1}\ldots \omega_0 = w$ we write
$$\omega = ... w|v...$$
and say that $|$ \textit{denotes the position of the origin}. We
equip $\mathcal{A}$ with the  discrete topology and
$\mathcal{A}^{\ZZ}$ with the  product topology.
A pair $(\varOmega,T)$ is called a
\textit{subshift} over $\mathcal{A}$ if $\varOmega$ is a closed
subset of $\mathcal{A}^{\ZZ}$ which is invariant under the
\textit{shift transformation}
$$T :
\mathcal{A}^{\ZZ} \longrightarrow \mathcal{A}^{\ZZ}, \;\: (T \omega)
(n):=\omega (n+1).$$
Whenever $\omega$ is a word over $\mathcal{A}$ (finite or infinite,
indexed by $\NN$ or by $\ZZ$) we define
$$\mbox{Sub} (\omega) := \mbox{Finite subwords of $\omega$}.$$
By convention, the set of finite subwords includes the empty word.
Every subshift  $(\Omega,T)$ comes naturally with the set
$\mbox{Sub} (\Omega)$ of associated finite words given by
 $$ \mbox{Sub} (\Omega)
:= \bigcup_{\omega\in \Omega} \mbox{Sub}(\omega).$$ A word $v \in
\mbox{Sub} (\Omega) $ is said to \textit{occur with bounded gaps} if
there exists an $L_v >0$ such that every $w\in \mbox{Sub} (\Omega)$
with $|w|\geq L_v$ contains a copy of $v$. As is well known (and not
hard to see)  $(\Omega,T)$ is minimal  if and only if every $v\in
\mbox{Sub} (\Omega)$ occurs with bounded gaps.  We will be concerned
with  a  strengthening of the bounded gaps condition:  A subshift
$(\varOmega,T)$ is called \textit{linearly repetitive} (LR), if
there exists a constant $C>0$ such that every word $v \in
\mbox{Sub}(\varOmega)$  occurs  in every word $w\in
\mbox{Sub}(\varOmega)$ of length at least $C |v|$. This notion has
been discussed under various names (including  linear recurrence) in
various contexts \cite{DHS,LP,Sol}. A linearly repetitive subshift
is uniquely
 ergodic and minimal \cite{Dur} and even allows for a uniform
 subadditive ergodic theorem \cite{Len3}.



\subsection{Basic features of the substitution  $\tau$}
 Consider the alphabet $\mathcal{A} = \{a,x,y,z\}$ and
let  $\tau$ be  the substitution mapping $a\mapsto a x a$, $x\mapsto
y$, $y\mapsto z$, $z\mapsto x$. Let $\mbox{Sub}_\tau$ be  the
associated set of finite words given by
$$\mbox{Sub}_\tau =\bigcup_{s\in \mathcal{A}, n\in \NN\cup\{0\} } \mbox{Sub}(\tau^n
(s)).$$ Then, the following three properties obviously hold:

\begin{itemize}

\item The letter $a$ is a prefix of $\tau^n (a)$ for all $n\in
\NN\cup\{0\}$.

\item The lengths $|\tau^n (a)|$ converge to $\infty$ for $n\to
\infty$.

\item Every letter of $\mathcal{A}$ occurs in $\tau^n(a)$ for some  $n$.

\end{itemize}

Since  $a$ is a prefix of $\tau (a)$  we have that  $\tau^n (a)$ is
a prefix of $\tau^{n+1} (a) $ for all $n\in\NN\cup\{0\}$. As the
length of $\tau^n (a)$ goes to infinity, there exists then  a unique
one-sided infinite word $\eta$ such that $\tau^n (a)$ is a prefix of
$\eta$ for all $n\in\NN\cup\{0\}$. This $\eta$ is a fixed point of
$\tau$
 i.e. $\tau (\eta) = \eta$. We will refer to it as \textit{the fixed
 point of the substitution $\tau$}.
 By the
third property we  have
$$\mbox{Sub}_\tau = \mbox{Sub} (\eta).$$
We can now associate to $\tau$ the subshift
$$\varOmega_\tau:=\{ \omega\in \mathcal{A}^\ZZ : \mbox{Sub} (\omega)\subset
\mbox{Sub}_\tau\}.$$ It can be easily seen that every other letter
of $\eta$ is the letter  $a$. In particular, the letter  $a$ occurs
in $\eta$ with bounded gaps. This implies that every word of
$\mbox{Sub}_\tau$ occurs with bounded gaps (as any such  word is a
subword of $\tau^n (a)$ and $\eta$ is a fixed point of $\tau$). For
this reason $(\varOmega_\tau,T)$ is minimal and $\mbox{Sub}(\omega)
= \mbox{Sub}_\tau$ for all $\omega \in \varOmega_\tau$. From Theorem
1 of \cite{DL2} (or a simple direct argument based on the above
three properties) we then obtain the following.

\begin{Theorem}\label{Theorem-tau-linear-repetitive} The subshift $(\varOmega_\tau,T)$ is linearly
repetitive.
\end{Theorem}


Our further considerations will be based on a more careful study of
the words $\tau^n (a)$. We set
$$p^{(0)}:= a\;\:\mbox{and}\;\:
 p^{(n)} :=\tau^n (a)\;\: \mbox{for}\;\: n\in \NN.$$
  A direct calculation gives
$p^{(n+1)} = \tau^{n+1} (a) = \tau^n (axa) =\tau^n (a) \tau^n
(x)\tau^n (a),$ i.e.
$$ (RF)\;\:\hspace{1cm}  p^{(n+1)} =  p^{(n)}  \tau^n (x)  p^{(n)}\;\:\mbox{with}\;\:  \tau^n (x) =\left\{ \begin{array}{ccc}  x &:& n = 3 k, k\in \NN\cup \{0\} \\ y &:& n = 3k+1, k\in \NN\cup\{0\}\\
z&:& n = 3 k + 2, k \in \NN \cup\{0\}  \end{array}, \right. $$ We
will refer to $(RF)$ as the \textit{recursion formula for the words
$p^{(n)}$}.\label{recursion}

We will now define three special elements
$\omega^{(x)},\omega^{(y)},\omega^{(z)}\in\varOmega_\tau$ closely
related to $\eta$, that will be useful in our subsequent analysis of
the subshift.

\begin{Lemma}[The special words $\omega^{(x)},\omega^{(y)},\omega^{(z)}$]\label{omega-x-y-z} For every $n\in \NN\cup\{0\}$ and every letter  $s\in\{x,y,z\}$ the word $p^{(n)} s
p^{(n)}$ occurs in $\eta$. In particular, for all $s\in\{x,y,z\}$
there exists a unique  element $\omega^{(s)}\in\varOmega_\tau$  such
that $$\omega^{(s)} = ... p^{(n)} s | p^{(n)}...$$ holds for all
natural numbers $n$,  where the $|$ denotes the position of the
origin.
\end{Lemma}
\begin{proof}
Note that $\tau^3 (a)$ contains $a x a$, $a y a$ and $a z a$. As
$\eta$ is a fixed point of $\tau$ and $\tau^n$ is injective on
$\{x,y,z\}$ the first statement follows. By the recursion formula
(RF) each $p^{(n+1)}$ starts and ends with $p^{(n)}$. Given this,
the second statement follows from the first statement.
\end{proof}

The  sequences $\omega^{(x)},\omega^{(y)},\omega^{(z)}$  are
different but all agree on $\NN$. Hence $\varOmega_\tau$ is not
periodic and this persist even after renaming the letters (provided
not all are given the same name).

\begin{Proposition}\label{Proposition-Shift-non-periodic} Let $\mathcal{B}$ be a finite set and  $ C : \mathcal{A} \longrightarrow
\mathcal{B}$ a map such that $C (x) = C(y) = C(z)$ does not hold.
Then, the subshift $\varOmega^{(C)}_\tau:=\{ C\circ \omega : \omega
\in \varOmega_\tau\}$ is not periodic and linearly repetitive.
\end{Proposition}
\begin{proof} Obviously, the map $\varOmega_\tau\longrightarrow \varOmega^{(C)}_\tau,
 \omega \mapsto C\circ \omega,$ is
continuous and onto and preserves linear repetitivity. In
particular, $(\varOmega^{(C)}_\tau,T)$ is a minimal subshift.
Moreover, the two-sided infinite words $C\omega^{(x)},
C\omega^{(y)}, C\omega^{(z)}$ all agree  on $\NN$ but are not all
equal (due to the assumption on $C$). This can easily be seen to
imply that the subshift $\varOmega_C$ is not periodic.
\end{proof}

Recall that a non-empty  word $w = w_1\ldots w_l\in\mathcal{A}^\ast$
with $w_j\in \mathcal{A}$ is called a \textit{palindrome} if $w =
w_l \ldots w_1$. The recursion formula and a simple induction imply
that, for any $n\in\NN$,   the word  $p^{(n)}$ is a palindrome of
length $2^{n+1} -1$ and starts and ends with $p^{(k)}$ for every
$k\in \NN\cup\{0\}$ with $k\leq n$.

As every other letter of $\eta$ is the letter   $a$ and $\eta$ is a
fixed point of $\tau$ we immediately infer that there exists a
sequence $(r_j)_{j\in\NN}$ in $\{x,y,z\}$ such that for all
$n\in\NN\cup\{0\}$ the word $\eta$ has a (unique) decomposition as
$$ \eta = p^{(n)}  \tau^n (r_1)  p^{(n)} \tau^n (r_2) ....$$
This way of writing $\eta$ will be called the
\textit{$n$-decomposition of $\eta$} and the  sequence
$(r_j)_{j\in\NN}$ the \textit{derived sequence of $\eta$}. Note that
the combinatorial properties of the sequence $(\tau^n (r_j))_{j\in
\NN}$
 are exactly the same as the combinatorial properties  of the sequence $(r_j)_{j\in\NN}$ as $\tau^n$
 is injective on $\{x,y,z\}$. Among these combinatorial properties
 we note the following.

\begin{Proposition}\label{Isolation-eta} In the derived sequence $(r_j)_{j\in\NN}$ the  letters $y$ and $z$
always occur  isolated preceded and followed by an $x$. The letter
$x$ always occurs either isolated (i.e. preceded and followed by
elements of $\{y,z\}$) or  in the form $x x x$. The analogue
statements hold for every natural number $n$ for the sequence
$(\tau^n (r_j))_j$ (with $x,y,z$ replaced by $\tau^n (x), \tau^n
(y)$ and $\tau^n (z)$).
\end{Proposition}
\begin{proof} As $\eta$ is a fixed point, we have
$$\eta = a r_1 a r_2 a ...= \tau (a) \tau (r_1) \tau (a) \tau (r_2) ...= p_1
\tau (r_1) p_1 \tau (r_2) ...$$ with $p_1 = \tau (a) = a x a$ and
$\tau(r_j)\in\{x,y,z\}$ for all $j\in\NN$. From this we find that
every other letter of $r$ must be the letter $x$. This shows the
claim on $y$ and $z$ and it also shows that $x$ appears in $r$
either isolated or with at least three $x$ in a row. Assume that
there is a block of the form $x x x x$ occurring in $r$. Then, $ a x
a x a x a x a l$ must occur in $\eta$ with $l\neq x$. In the
$1$-decomposition of $\eta$ this yields   $ p_1 x p_1 x p_1 l$. This
gives a contradiction when we consider the $2$-decomposition as $p_2
\neq p_1 x p_1$.  The last statement follows as $\tau^n$ is
injective on $\{x,y,z\}$.
\end{proof}




%

Our next aim is to provide an analogue of the $n$-decomposition of
$\eta$ for each $\omega\in\varOmega_\tau$.  We are thus looking for
each  $n\in\NN\cup\{0\}$ and $\omega\in\varOmega_\tau$ for a
decomposition of the form
$$\omega =  ... p^{(n)} s_0 p^{(n)} s_1 p^{(n)} s_2...$$
with  $s_k\in\{x,y,z\} $ for all $k\in \ZZ$ and  the origin
$\omega_0$ belonging to $s_0 p^{(n)}$. Such a decomposition will be
referred to as the \textit{$n$-decomposition} of $\omega$.
In such a decomposition the positions of the $s_j$'s are given by $q
+ 2^{n+1} \ZZ$ with $q\in \{0,\ldots, 2^{n+1}-1\}$. Thus, the
positions are given by an element of $\ZZ / 2^{n+1} \ZZ$. This
suggests the following definition.

\begin{Definition}[$n$-partition] For  $n\in \NN\cup\{0\}$ we call
an element $P\in  \ZZ / 2^{n+1} \ZZ $ an  $\textit{n-partition}$ of
$\omega\in \varOmega_\tau$ if for every $q\in P$ the following two
properties hold:
\begin{itemize}
\item $\omega_{q} \in \{x,y,z\}$;
\item $\omega_{q+1} \ldots \omega_{q+2^{n+1} -1} = p^{(n)}$.
\end{itemize}

\end{Definition}
Clearly, for $\omega\in\varOmega_\tau$,  the existence (resp.
uniqueness) of an $n$-partition is equivalent to the existence
(resp. uniqueness) of an $n$-decomposition.  In this sense these two
concepts are equivalent.
 It is not apparent that such an $n$-partition (or an
$n$-decomposition)  exists at all. Here is our corresponding result.
\begin{Theorem}[Existence and Uniqueness of $n$-partitions]\label{theorem-n-partition}  Let $n\in\NN\cup\{0\}$ be given. Then every $\omega\in
\varOmega_\tau$ admits a unique $n$-partition $ P^{(n)} (\omega)$
and the map
$$ P^{(n)} : \varOmega_\tau \longrightarrow \ZZ / 2^{n+1} \ZZ, \; \omega \mapsto P^{(n)} (\omega),$$ is
continuous and equivariant (i.e. $P^{(n)} (T\omega) = P^{(n)}
(\omega) + 1$).
\end{Theorem}
\begin{proof} \textit{Existence of $P^{(n)} (\omega)$.} Recall that
$\omega^{(x)}\in\varOmega_\tau$ is the unique word with
$\omega^{(x)} = ...p^{(n)} x| p^{(n)}...$, where $|$ denotes the
position of the origin. Now, obviously, $\omega^{(x)}$ admits an
$n$-partition (by its very definition and $(RF)$). Moreover,  the
subshift is minimal. Hence, every $\omega\in \varOmega_\tau$ can be
approximated by a sequence of translates of $\omega^{(x)}$. These
all carry natural $n$-partitions coming from the $n$-partition of
$\omega^{(x)}$. As the values of these $n$-partitions all lie within
the finite set $\ZZ / 2^{n+1} \ZZ$ we can assume (after restricting
attention to a subsequence) without loss of generality that these
values are all equal. Now the existence of an $n$-partition for
$\omega$ is clear.

\textit{Uniqueness.} As discussed in the paragraph preceding the
theorem the concepts of $n$-decomposition and $n$-partition are
equivalent in the sense that the  existence (uniqueness)  of an
$n$-partition implies the existence (uniqueness) of an
$n$-decomposition and vice versa. This will be used in order to
obtain the uniqueness. Our proof proceeds by induction. The case $n
=0$ is clear. (In this case $p^{(0)} =a$.) Let us now show how to
proceed from $n$ to $n+1$. Consider an $(n+1)$-decomposition of
$\omega$. Such a decomposition exists by the already shown part.
Chose $s\in \{x,y,z\}$ with $p^{(n+1)} = p^{(n)} s p^{(n)}$. Then,
out of the $n+1$ decomposition of $\omega$ we obtain an
$n$-decomposition by just replacing $p^{(n+1)} $ by $p^{(n)} s
p^{(n)}$ in the corresponding decomposition of $\omega$. This
$n$-decomposition  is unique by our induction assumption. Now, it is
not hard to see that non-uniqueness of the $(n+1)$-decomposition can
only occur if $\omega$ is periodic. However, as discussed above in
Proposition \ref{Proposition-Shift-non-periodic} there is no
periodic sequence in $\varOmega$.

\textit{Continuity.} This is a direct consequence of uniqueness. Let
$\omega^{(k)}$ be a sequence converging to $\omega$ and let $P_k$
and $P$ be the respective $n$-partitions. We have to show $P_k \to
P$. As the space $\ZZ / 2^{n+1} \ZZ$ is finite (hence compact), it
suffices to show that any converging subsequence of $(P_k)$
converges to $P$. Now, it is clear that whenever a subsequence of
$(P_k)$ converges to some $P^{'}$ then $P^{'}$ is an $n$-partition
of $\omega$. By uniqueness we infer $P^{'} =P$ and this gives the
desired statement.

\textit{Equivariance.} This is a direct consequence of the existence
and uniqueness statements.
\end{proof}
\begin{Corollary}\label{Isolation} Let $n\in\NN\cup \{0\}$ and
$\omega \in \varOmega_\tau$ be arbitrary and consider the
$n$-decomposition  $\omega =  ... p^{(n)} s_0 p^{(n)} s_1 p^{(n)}
s_2...$. Then, the letters $\tau^n (y), \tau^n (z)$  occur isolated
in the sequence $(s_j)$ and the letter $\tau^n (x)$ is either
isolated or arises  in the form $\tau^n (x) \tau^n (x)\tau^n (x)$.
\end{Corollary}
\begin{proof} By Proposition
\ref{Isolation-eta} the statement is true for the $n$-decomposition
of $\eta$. By construction it then easily  follows   for the
$n$-decomposition of $\omega^{(x)}$ as well. By minimality   any
$\omega\in \varOmega_\tau$ can be approximated by elements  of the
form  $T^k \omega^{(x)}$ with suitable $k\in\ZZ$.  Now continuity
and equivariance  of the $n$-partition shown  in the previous
theorem easily give  the desired statement.
\end{proof}



\subsection{Palindromes and reflection symmetry in
  $\mbox{Sub}_\tau$ }\label{section-reflection}
In this section we  study  palindromes in $\mbox{Sub}_\tau$  and the
corresponding reflection  symmetries.

For a finite word $w = w_1\ldots w_n$ we define the reflected word
$w^R$ by $w^R = w_n \ldots w_1$. There are two different ways to
extend  this operation to two  sided infinite words. One way is to
associate to a two-sided infinite sequence $\omega$ the   sequence
$\omega^R$ defined by $\omega^R (n) := \omega (-n)$. Thus, the
operation $R$ on the two sided infinite words is just the reflection
at the origin. The other way  is to associate to a two sided
infinite sequence $\omega$ the sequence $\widetilde{\omega}$ given
by $\widetilde{\omega} (n) := \omega (1-n)$. Thus, for $\omega =
...\omega_{-1} \omega_0 | \omega_1 \omega_2...$ we have
$\widetilde{\omega} = ...\omega_2 \omega_1 | \omega_0
\omega_{-1}...$ In this sense, $\widetilde{}$ is just the reflection
at $|$. Of course, both reflections are related. In fact, we have
$\widetilde{\omega} = T^{-1} \omega^R .$ We now turn to studying how
$\varOmega_\tau$ is compatible with the reflection operations $R$
and $\; \widetilde{}$.

\begin{Proposition} \label{Inversion} For every $\omega\in\varOmega_\tau$ the element
$\widetilde{\omega}$ also belongs to $\varOmega_\tau$, and the map
$$\varOmega_\tau\longrightarrow \varOmega_\tau,\;\omega\mapsto
\widetilde{\omega},$$ is a homeomorphism without fixed point.
\end{Proposition}
\begin{proof} As all $p^{(n)}$'s are palindromes, the set
$\mbox{Sub}_\tau$ is invariant under $R$, i.e. $R(\mbox{Sub}_\tau)
=\mbox{Sub}_\tau$. This easily gives the first statement. As exactly
one of the letters $\omega_0$ and $\omega_1$ is $a$, the map
$\widetilde{}$ cannot have a fixed point.
\end{proof}

By the previous proposition, the map  $\; \widetilde{}\; $ does not
have fixed points on $\varOmega_\tau$. However, there are  words
which are fixed points of   $R$. We say that $\omega$ is
\textit{symmetric around $p\in \ZZ$} if $\omega_{p + k} =
\omega_{p-k}$ for all $k\in \NN$ (i.e. if $T^p \omega$ is a fixed
point of  $R$). Then, clearly
 each $\omega^{(s)}$, $s\in\{x,y,z\}$,  is symmetric around $0$. In
 fact, the converse is valid as well  as shown in the next theorem.

\begin{Theorem}\label{Symmetric-word} Let $\omega\in\varOmega_\tau$ be symmetric around
$p\in  \ZZ$. Then, there exists an $s\in\{x,y,z\}$ with $\omega =
T^p \omega^{(s)}$.
\end{Theorem}
\begin{proof} Without loss of generality we can assume  $p =0$.
We consider the $n$-decomposition of $\omega$. We say that the
$n$-decomposition has a \textit{break point} at the origin if the
$n$-decomposition around the origin looks like $p^{(n)} s |p^{(n)}$
with $|$ denoting the position of origin.  It suffices to show that
the $n$-decomposition has a break point at the origin for every
$n\in\NN\cup \{0\}$. This is done by induction:

$n=0$: Consider the word $\omega_{-1}\omega_0\omega_1$ and note that
$\omega_{-1} = \omega_1$ must hold by symmetry. Assume that
$\omega_0  =a $ holds. As the letter  $a$ occurs isolated, we infer
that $\omega_{-1} = \omega_1$ cannot be $a$. Hence, it must be $x,y$
or $z$.  After deletion of all the letters  $a$ from $\omega$, the
letters $y,z$ occur isolated and $x$ occurs either isolated or with
power $3$  by Corollary \ref{Isolation}. Therefore we get  a
contradiction to $\omega_0 = a$.  Set $s:=\omega_0\in\{x,y,z\}$.
Then,  $\omega_{-1}\omega_0\omega_1= as | a$ and the
$0$-decomposition has a break point.

$n\Longrightarrow n+1$: As the $n$-decomposition of $\omega$ has a
break point at the origin and $\omega$ is symmetric, the
$n$-decomposition of $\omega$ around the origin looks like
$$\omega = ... t \,  p^{(n)} s |p^{(n)}\,  t ...$$
(with a suitable $t\in\{x,y,z\}$). We have to  show that $p^{(n)} s
|p^{(n)}$ does not become  $p^{(n+1)}$ in the $(n+1)$-decomposition
of $\omega$. Assume the contrary. Then the $(n+1)$-decomposition of
$\omega$ looks like
$$\omega = ... t\,  p^{(n+1)} \, t ...$$
and  the  letter $t$ is not isolated in the $(n+1)$-decomposition.
Hence, by Corollary \ref{Isolation},  the letter $t$ occurs with a
third power and the $(n+1)$-decomposition of $\omega$ looks like
$$\omega = t p^{(n+1)} t (p^{(n)} s |p^{(n)}) t p^{(n+1)} v$$
or
$$\omega =v p^{(n+1)} t (p^{(n)} s |p^{(n)}) t p^{(n+1)} t$$
with $v\in \{x,y,z\}$ with $v\neq t$, where we have written
$(p^{(n)} s | p^{(n)})$ instead  of $p^{(n+1)}$ to denote the
position of the origin. This, however, is a contradiction to the
symmetry of $\omega$.
\end{proof}

\begin{Corollary}\label{Orbital-uniqueness} Let
$\omega\in\varOmega$ be given. Then, the orbits $\{T^n \omega :
n\in\ZZ\}$ and $\{T^n \widetilde{\omega} :n\in \ZZ\}$ are disjoint
 unless $\omega = T^p \omega^{(s)}$ for some  $s\in\{x,y,z\}$ and  $p\in\ZZ$.
\end{Corollary}
\begin{proof}
It is not hard to see that the two orbits in question can only
intersect if $\omega$ is symmetric. Thus, the statement of the
corollary follows from the previous theorem.
\end{proof}

\subsection{Powers in $\mbox{Sub}_\tau$}
In this section we study  powers in  $\mbox{Sub}_\tau$ and determine
the index (i.e. the   supremum over all powers).




Our first result gives the existence of three-blocks followed by a
(long) prefix of the same  block. Thus, it 'almost' gives the
existence of a four-block.

\begin{Lemma}[Almost four-blocks in  $\omega^{(s)}$] \label{Three-blocks-in-omega-s} Let $s\in \{x,y,z\}$ be
given and let $\omega^{(s)}$ be  the unique word with $\omega^{(s)}
= ... p^{(n)} s|p^{(n)}...$ for all $n$ (where $|$ denotes the
position of the origin). Then,
$$\omega^{(s)} = ...p^{(3n + k)} s p^{(3n +k)} s| p^{(3n +k)} s p^{(3n +k)}...$$
for all $n\in \NN\cup \{0\}$, where $k=0$ for $s = x$, $k=1$ for $s
= y$ and $k=2$ for $s = z$.
\end{Lemma}
\begin{proof} From the definition of
$\omega^{(s)}$ we have $\omega^{(s)} = ...p^{(3 n + 3)} s| p^{(3n
+3)}$ for all $n\in\NN$. Now, the lemma follows after we apply  the
recursion formula $(RF)$ from page~\pageref{recursion} the
corresponding number of times.
\end{proof}

As shown in the next lemma, similar almost four-block structures
occur in $\lambda$-almost every element of $\varOmega_\tau$, where
$\lambda$ is the unique $T$-invariant probability measure on
$\varOmega_\tau$.

\begin{Lemma}[Almost four-blocks in almost every $\omega$] \label{Three-blocks-in-almost-every-omega} For
$\lambda$-almost every $\omega \in\varOmega_\tau$ there exist
sequences of words $(w_n)$, $(v_n)$ with $v_n$ prefix of $w_n$ for
each $n\in \NN$  and $|w_n| \to \infty$ and $\frac{|v_n|}{|w_n|}\to
1$ and  $\omega = ... w_n w_n | w_n v...$
\end{Lemma}
\begin{proof} From the previous lemma (with $k = n =0$) we infer
that the word $axaxaxa$ belongs to $\mbox{Sub}_\tau$ i.e. the word
$u = w^3 v$ with $w = ax$ and $v = a$ occurs in $\eta$. Thus, all
words of the form $\tau^n (w^3 v)$ will occur in $\eta$ as well.
 Clearly,  $|\tau^n (w)| = |\tau^n (a)| + |\tau^n (x)| = |\tau^n (v)| +
 1$ and $|\tau^n (w)|\to\infty$ and $\frac{|\tau^n (v)|}{|w|}\to 1$
 follows. Now, the  desired statements follows from  Lemma 4.2 in \cite{DL2}.
The Lemma 4.2 of \cite{DL2} only deals with almost sure existence of
$w_n$ with the desired properties. However, close inspection of the
proof shows existence of $v_n$ as well.
\end{proof}

The preceding  results show  that there is quite a supply of
three-blocks at hand for elements of $\varOmega_\tau$. This will be
used in the proof of one of  our main results  on  spectral theory.
The preceding results also naturally lead to the question whether
four blocks are present. It turns out that this is not the case.
While this will not be used in the proofs of our main results we
next include a discussion.





\begin{Proposition} \label{Prop-Alignment}
Consider a natural number $n$ and $s\in \{x,y,z\}$. If $p^{(n)} s
p^{(n)}$ occurs in $\eta$ at the position $l$ (i.e. $\eta_l
\eta_{l+1} \ldots \eta_{l + | p^{(n)} s p^{(n)}  | -1} = p^{(n)} s
p^{(n)}$ holds), then $l$ is of the form $1 + k 2^{n+1}$ for some
$k\in\NN\cup\{0\}$. This means that if $ p^{(n)} s p^{(n)}$ occurs
somewhere in $\eta$ then both of its words $p^{(n)}$ actually agree
with blocks $p^{(n)}$ appearing in the $n$-decomposition $\eta =
p^{(n)} r_1^{(n)} p^{(n)} r_2^{(n)} p^{(n)} ....$
\end{Proposition}
\begin{proof} This follows by induction on $n$. The case $n=0$ is
clear. Let us assume that the statement holds for $n$. Consider an
occurrence of
$$ w = p^{(n+1)} s p^{(n+1)} = p^{(n)} t p^{(n)} s p^{(n)} t p^{(n)}$$
with suitable letters $s$ and $t$ in $\{x,y,z\}$. By assumption the
$p^{(n)}$ are well aligned with the $n$-decomposition.

\smallskip

\textit{Case 1: $s \neq t$.} By  $s \neq t$,  both $p^{(n)} t
p^{(n)}$ in the above formula for $w$  will become a $p^{(n+1)}$ in
the $(n+1)$-decomposition. This readily gives the statement.

\smallskip

\textit{Case 2: $s=t$.} By Corollary \ref{Isolation}, there can be
no more than three occurrences of $s =t$ in a row. Thus, again the
first and the last $p^{(n)} t p^{(n)}$ will become a $p^{(n+1)}$ in
the $(n+1)$-decomposition and the desired statement follows.
\end{proof}

If $w$ is a finite word in $\mbox{Sub}_\tau$ and $v$ is a prefix of
$w$ and $N$ is a natural number we define the \textit{index of the
word $w$ in $w^N v$} by $N +  \frac{|v|}{|w|}$ and denote it by $Ind
(w, w^N v)$. We then define the \textit{index of the word $w$}  by
$$Ind (w):=\max\{ Ind(w,w^N v) : \mbox{$v$  prefix of $w$}, N\in \NN, w^N
v\in\mbox{Sub}_\tau\}.$$ As our subshift is minimal and aperiodic
the index of every word can easily be seen to be finite. We define
the \textit{index of the subshift} (also known as \textit{critical
exponent})  as
$$Ind (\varOmega_\tau) =\sup \{Ind (w) : w\in\mbox{Sub}_\tau\}.$$

\begin{Theorem}[Index of $\varOmega_\tau$]
The sequence $\eta$ does not contain a fourth power i.e. for every
$w\in\mbox{Sub}_\tau$ the inequality $Ind (w) < 4 $ holds.  In
particular, we have $ 4 =  Ind (\varOmega_\tau)$.
\end{Theorem}

\begin{proof} We  show that $\eta$ does not contain a fourth power.
The statement on the index is then a direct consequence from Lemma
\ref{Three-blocks-in-almost-every-omega}.

We consider the index of a word $w\in \mbox{Sub}_\tau$. We first
consider the case $|w|\leq 3$. If $|w| = 1$ or $|w| = 3$ then $w$
either starts and ends with an $a$ or starts and ends with a letter
which is not $a$. In both cases $w w$ can not occur as exactly every
other letter of $\eta$ is an $a$. If $w = 2$ we have $w\in
\{ax,ay,az, xa,ya,za\}$ and from Corollary \ref{Isolation} we infer
that $w w w w $ can not occur. We now consider the case $ |p^{(n)}|
+1 = 2^{n+1} \leq |w|\leq |p^{(n+1)}|$ for some $n\geq 1$. Assume
that $w w w$ occurs in $\mbox{Sub}_\tau$.

\smallskip

\textit{Claim.} The length of $w$ is given by $|w|= |p^{(n)}| +1 =
2^{n+1}$.

\textit{Proof of Claim.} Consider the $n$-decomposition of $\eta$ at
an occurrence of $www$. By 'chopping off' a suffix  of $ w w w $ of
length less than $|p^{(n)}| < |w|$ we infer that there exists a word
$v$ of the same length as $w$ such that $v v$ occurs in $\eta$ and
starts exactly at the beginning of a $p^{(n)}$ of the
$n$-decomposition of $\eta$ i.e. such that $ v v$ is a prefix  of a
part of the $n$-decomposition given by
$$ p^{(n)} r_k^{(n)} p^{(n)} r_{k+1}^{(n)}  p^{(n)} r_{k+2}^{(n)} p^{(n)}.$$
Consider the $n-1$ decomposition of this word. It is given by
$$ p^{(n-1)} s p^{(n-1)} r_k^{(n)} p^{(n-1)} s p^{(n-1)} r_{k+1}^{(n)}
p^{(n-1)} s p^{(n-1)} r_{k+2}^{(n)} p^{(n-1)} s p^{(n-1)}$$ with
$s\in\{x,y,z\}$ suitable. From Proposition \ref{Prop-Alignment}
(applied with $p^{(n)}$ instead of $p^{(n+1)}$) we then infer that
the words $v$ must be well-aligned with the occurrences of the
$p^{(n-1)}$. This gives $|v| =|p^{(n)}| +1$ or $|v| =|p^{(n)}| +1 +
|p^{(n-1)}|$. The second case can easily be seen to yield $s = r_k =
r_{k+1} = r_{k+2}$ and this contradicts Proposition \ref{Isolation}.
This proves the claim.

Assume now that $w$ with $|p^{(n)}| +1 = 2^{n+1} \leq |w|\leq
|p^{(n+1)}|$ is such that $ w w w w $ occurs in $\eta$. By the
claim, we then have  $|w| =|p^{(n)}| +1$. Considering the
$n$-partition of $\eta$ we infer the existence of a  fourth power of
a single letter in the sequence $r^{(n)}$ contradicting  Corollary
\ref{Isolation}.
\end{proof}


The remaining two parts of this section  are  not used in the proofs
of our main results on spectral theory. Instead they  contain some
discussion of further properties of the subshift $\varOmega_\tau$.
In this context we also note that it is possible to determine the
word complexity of the subshift  based on the decompositions
provided above. Details are given in \cite{GLN-survey,GLN-note}.

\subsection{Generating the fixed point $\eta$ by an automaton}
In this section we show that the fixed point of the substitution can
be generated by an automaton. This fits to the  general theory on
how to exhibit fixed points of substitutions by automata, see e.g.
the monograph \cite{AS} to which we also refer for background on
automata.

\begin{figure}[h!]
\begin{center}
\hspace*{-1cm}
\includegraphics[scale=0.3]{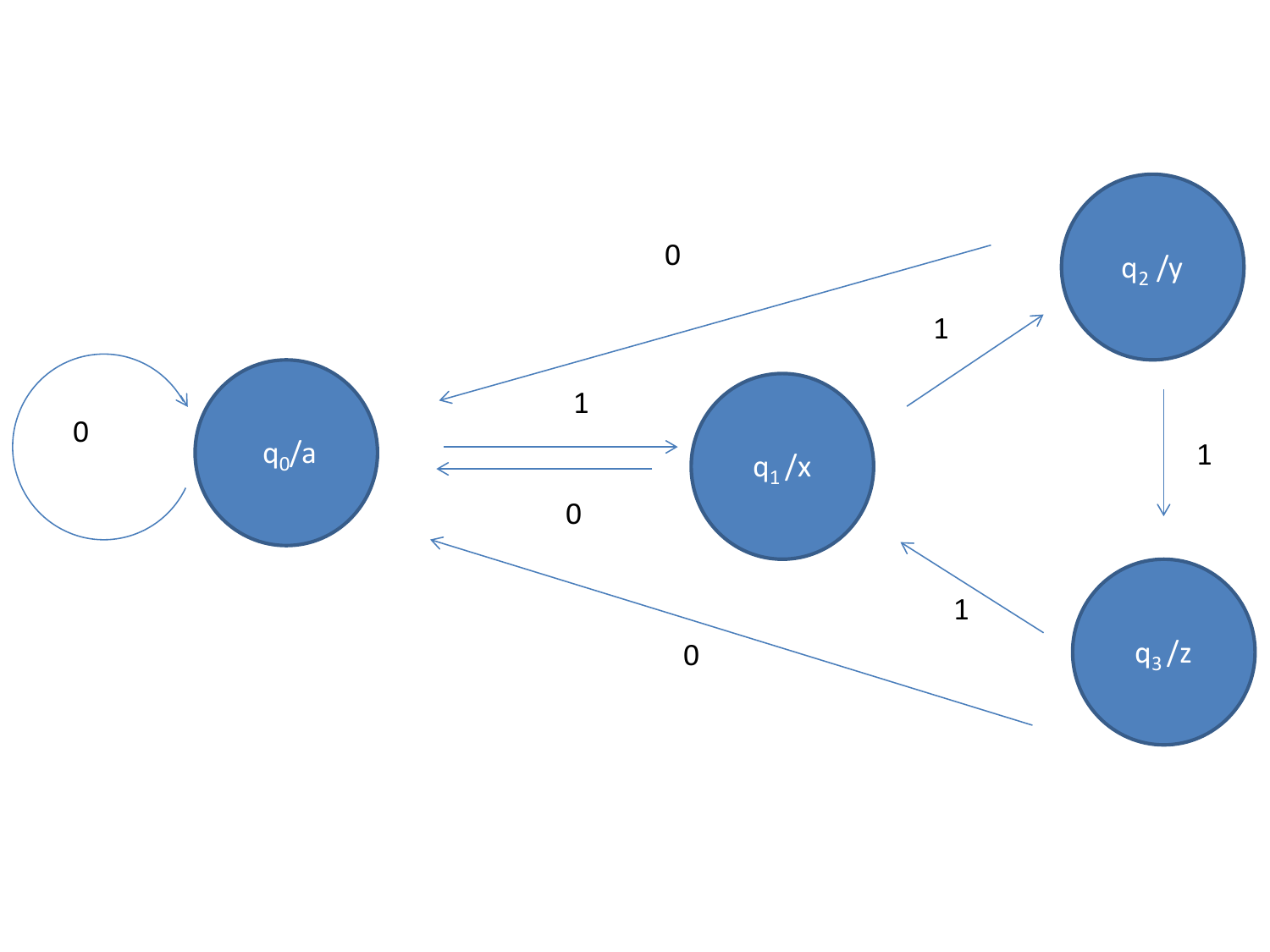}
\hspace*{-1cm}
 \caption{The automaton generating $\eta$}
\label{automaton}
\end{center}
\end{figure}

Consider the automaton $\mathcal{A}$ from figure \ref{automaton}. It
is an automaton over the alphabet $\{0,1\}$ with  four states
$q_0,q_1,q_2,q_3$  labeled by $a,x,y,z$ respectively. Then, the
infinite  sequence
$$\mathcal{A}_{q_0} : \NN\cup\{0\}\longrightarrow \{a,x,y,z\}$$
generated by the automaton with initial state $q_0$ is defined as
follows: Write $n\in \NN\cup \{0\}$ in its binary expansion as
$$n = x_0 2^i + x_1 2^{i-1} + \cdots + x_{i-1} 2 + x_i$$
with $i\in \NN\cup \{0\}$ and $x_j \in \{0,1\}$, $j = 0, \ldots, i$.
Consider now the path $p_n$ in the automaton  starting in $q_0$ and
following the sequence $x_0 x_1\ldots x_i$. Then, $\mathcal{A}_{q_0}
(n)$ is defined to be the label of the state where this path ends.
\begin{Theorem} The fixed point $\eta$ of $\tau$ agrees with
$\mathcal{A}_{q_0}$ (where the fixed point is considered as a map
from $\NN\cup \{0\}$ to $\{a,x,y,z\}$).
\end{Theorem}

This theorem is an immediate consequence of the next proposition. To
state the proposition we will need some further pieces of notation.
For each $n\in \NN\cup \{0\}$ and each state $q$ of the automaton we
define $f^{(n)} (q)$ to be  the   word over $\{a,x,y,z\}$ of length
$2^{n}$ obtained  in the following way: Let $v_1,\ldots, v_{2^n}$ be
the list of  all words of length $n$  over $\{0,1\}$ in
lexicographic order (where $0< 1$). Consider now for each $k =
1,\ldots, 2^n$ the path  in the automaton starting at $q$ and
following the word $v_k$. Then, the $k$-th letter of $f^{(n)}(q)$ is
defined to be the label of the state where this path ends.

\begin{Proposition} We have $ f^{(n+1)} (q_i) = p^{(n)} \tau^{n+i}
(x)$ for each $n\in\NN$ and $i\in \{0,1,2,3\}$.
\end{Proposition}
\begin{proof} This is proven by induction (see \cite{GLN-survey} for further details as well).  The case  $n=1$ follows by inspection.
Assume now that the statement is true for some $n\geq 1$ and
consider $n+1$. The lexicographic ordering of the words of length
$n+2$ over $\{0,1\}$ is given by $ 0 v_1,\ldots 0 v_{2^{n+1}}$, $1
v_1,\ldots 1, v_{2^{n+1}}$,  where $v_1,\ldots, v_{2^{n+1}}$ is the
lexicographic ordering of the words of length $n+1$ over $\{0,1\}$.
From the rules of the automaton we  then obtain   $f^{(n+2)} (q_i) =
f^{(n+1)} (q_0) f^{(n+1)} (q_{i+ 1})$ for each $i = 0,1,2,3$,  where
we set $q_{4} = q_1$. The assumption for $n$ and the recursion now
imply the desired statement.
\end{proof}

\subsection{Replacing  $\tau$ by a primitive substitution}
\label{section-recoding}  The substitution $\tau$ arises naturally
in the study of Grigorchuk groups $G$ and its Schreier graphs (see
below). From the point of view of subshifts it has the  disadvantage
of not being primitive. It turns out that it is possible to find a
primitive substitution $\xi$ with the same fixed point - and hence
the same subshift - as $\tau$. The material presented here was
pointed out to us by Fabien Durand \cite{Dur2}.

Consider the substitution $\zeta$ on the alphabet $\{a,x,y,z\}$ with
$$\zeta (a) = ax, \zeta(y) = a y,\zeta (y) = a z, \zeta (z) = ax (= \zeta
(a)).$$ It is not hard to see that this is a primitive substitution.
\begin{Proposition} For any natural number $n$  the equality
$\zeta^n (a) = \tau^{n-1} (a) \tau^{n-1} (x)$ holds. In particular,
the fixed point $\eta$ of $\tau$ agrees with the fixed point of
$\zeta$ and the subshift $(\varOmega_\zeta,T)$ generated by  $\zeta$
agrees with $(\varOmega_\tau,T)$.
\end{Proposition}
\begin{proof} This follows rather directly by  induction. A discussion including further
details can also  be found in \cite{GLN-survey}.
\end{proof}

Subshifts associated to primitive substitutions are linearly
repetitive \cite{Dur,DZ}. Thus,  one can base an  alternative proof
of Theorem \ref{Theorem-tau-linear-repetitive} on the preceding
result. Also, as $\zeta$ has constant length (i.e. the length of
$\zeta (t)$ is the same for any $t$) and $\zeta(t)$ starts
with $a$ for any  $t$, a result of Dekking \cite{Dek} implies
purely discrete spectrum as well as that the  so-called maximal
equicontinuous factor is the binary odometer, see \cite{GLN-survey}
for a more detailed exposition.

\section{Connecting the dynamical system $(X,G)$     with
the subshift $(\varOmega_\tau,T)$}\label{Connecting} In this section
we will link the Schreier graphs of {Grigorchuk's} group $G$ and the
subshift $\varOmega_\tau$ in a precise way.
Throughout this section we will use the alphabet $\mathcal{A} =
\{a,x,y,z\}$ and the alphabet $\mathcal{B} = \{a,b,c,d\}$. Moreover,
we will denote the metric  space of isomorphism classes of  rooted
connected  graphs with labels in $\mathcal{B}$ by $\mathcal G_*
(\mathcal{B})$ (see Section \ref{Schreier-finite} as well).

We will approximate infinite graphs by finite graphs and describe
this approximation via finite words approximating infinite words. To
phrase this conveniently, it  will be useful for us to equip the set
of all words (finite and infinite) over $\mathcal{A}$ with a
topology. To do  so we will extend elements of $\mathcal{A}^\ast$
and $\mathcal{A}^\NN$ to functions on $\ZZ$ which take an additional
value $\star$ at those places where they are not originally defined.
Specifically, choose an element $\star$ which does not belong to
$\mathcal{A}$ and consider the new alphabet
$\mathcal{A}\cup\{\star\}$ and equip it with the discrete topology.
Then, $(\mathcal{A}\cup\{\star\})^\ZZ$ is a compact set in the
product topology.  For a function  $ \omega: \ZZ \longrightarrow
\mathcal{A}\cup\{\star\}$ we define its \textit{support},
$\mbox{supp}(\omega)$, via
$$\mbox{supp}(\omega):=\omega^{-1}(\mathcal{A}).$$
We then call the  elements of
$$\word(\mathcal{A}) :=\{\omega \in (\mathcal{A}\cup\{\star\})^\ZZ :
\mbox{supp}(\omega) \mbox{ is an interval}\}$$ the   \textit{words
associated to $\mathcal{A}$}.  Here, a subset $I$ of $\ZZ$ is called
an  \textit{interval} if with $a,b\in I$ also all $c\in \ZZ$ with
$a\leq c \leq b$ belong to $I$.  Clearly $\word(\mathcal{A})$ is a
closed   subset of $(\mathcal{A}\cup\{\star\})^{\ZZ}$ which is
invariant under the shift $T$. Hence, $\word(\mathcal{A})$ is
compact and $(\word(\mathcal{A}),T)$ is a subshift (over the
alphabet $\mathcal{A}\cup\{\star\}$).  The elements of
$\mathcal{A}^\ast$ and $\mathcal{A}^\NN$ can be canonically
 identified with elements of $\word(\mathcal{A})$ by extension by
$\star$. More specifically, we will identify $w \in
\mathcal{A}^\ast$ with the function $\omega_w : \ZZ \longrightarrow
(\mathcal{A}\cup\{\star\})$ defined by $\omega (w) (n) = w_n$ for
$n\in \{1,\ldots, |w|\}$ and $\omega (w) (n) = \star$ otherwise.
Similarly, we will identify  $ \xi \in \mathcal{A}^\NN$ with the
function $\omega_\xi : \ZZ \longrightarrow
(\mathcal{A}\cup\{\star\})$ defined by $\omega_\xi (n) = \xi(n)$ for
$n\in \NN$ and $\omega_\xi (n) = \star$ otherwise. These
identifications will be tacitly assumed in the sequel.


\subsection{The substitution $\Theta$}
In this section we present the graph version of $\tau$.

In Section \ref{Section-Schreier} we have seen that the action of
$G$ on the $n$-th level of the
 binary tree
gives rise to the  $n$-th level Schreier graphs $\varGamma_n$, $n\in
\NN$ (see Figure \ref{Schreier-finite} as well).  As discussed in
\cite{BG},  the substitutional rules given in Figure
\ref{picture-substitution} describe how to construct recursively the
graph $\varGamma_{n+1}$ from $\varGamma_n$, starting from the
Schreier graph of the first level $\varGamma_1$. Specifically, the
construction consists in replacing the labeled subgraphs of
$\Gamma_{n}$ on the top of Figure \ref{picture-substitution}  by the
new labeled graphs given on the bottom of Figure
\ref{picture-substitution}:

\begin{figure}[h!]
\begin{center}
\begin{picture}(400,110)
\letvertex A=(65,100)\letvertex B=(105,100)\letvertex C=(145,100)
\letvertex D=(185,100)\letvertex E=(225,100)\letvertex F=(265,100)\letvertex G=(305,100)
\letvertex H=(345,100)\letvertex I=(45,20)\letvertex L=(65,20)\letvertex M=(105,20)\letvertex N=(125,20)
\letvertex c=(145,20)
\letvertex d=(185,20)\letvertex e=(225,20)\letvertex f=(265,20)\letvertex g=(305,20)
\letvertex h=(345,20)

\put(82,60){$\Downarrow$}\put(162,60){$\Downarrow$}\put(242,60){$\Downarrow$}\put(322,60){$\Downarrow$}

\put(62,92){$u$} \put(102,92){$v$}\put(142,92){$u$}
\put(182,92){$v$}\put(222,92){$u$} \put(262,92){$v$}
\put(302,92){$u$}\put(342,92){$v$}

\put(40,10){$1u$} \put(62,10){$0u$}\put(102,10){$0v$}
\put(122,10){$1v$}

\put(141,10){$1u$} \put(181,10){$1v$}\put(221,10){$1u$}
\put(261,10){$1v$} \put(301,10){$1u$}\put(341,10){$1v$}

\drawvertex(A){$\bullet$}\drawvertex(B){$\bullet$}
\drawvertex(C){$\bullet$}\drawvertex(D){$\bullet$}
\drawvertex(E){$\bullet$}\drawvertex(F){$\bullet$}
\drawvertex(G){$\bullet$}\drawvertex(H){$\bullet$}
\drawvertex(I){$\bullet$}\drawvertex(L){$\bullet$}
\drawvertex(M){$\bullet$}\drawvertex(N){$\bullet$}
\drawvertex(c){$\bullet$}\drawvertex(d){$\bullet$}
\drawvertex(e){$\bullet$}\drawvertex(f){$\bullet$}
\drawvertex(g){$\bullet$}\drawvertex(h){$\bullet$}

\drawundirectedloop(L){$d$}\drawundirectedloop(M){$d$}

\drawundirectededge(A,B){$a$}\drawundirectededge(C,D){$b$}
\drawundirectededge(E,F){$c$} \drawundirectededge(G,H){$d$}
\drawundirectededge(c,d){$d$}\drawundirectededge(e,f){$b$}
\drawundirectededge(g,h){$c$}
\drawundirectededge[r](L,I){$a$}\drawundirectededge(M,N){$a$}
\drawundirectedcurvededge(L,M){$b$}\drawundirectedcurvededge[b](L,M){$c$}

\end{picture}
\end{center}
\caption{The substitution $\varTheta$ \label{picture-substitution}.}
\end{figure}
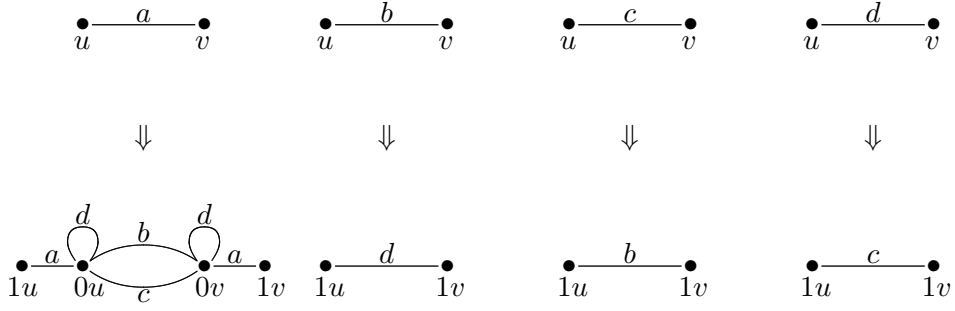



The  substitution rules and an easy induction directly give
 that for every natural number $n$  the graph
$\varGamma_n$ has 'linear structure' with rightmost vertex  given by
$1^n$ and this vertex 'becomes' the rightmost vertex $1^{n+1}$ under
the substitution. The leftmost vertex is given by $1^{n-1} 0$.

These rules, allowing to proceed from $\varGamma_n$ to
$\varGamma_{n+1}$, suggest to study the substitution $\Theta$ acting
in the following way on the set $\mathcal G_* (\mathcal{B})$ (see
Figure \ref{picture-substitution} disregarding the notation under
the vertices):

\begin{itemize}

\item it keeps the root;

\item it  replaces the edges labeled by $b$ with edges labeled by
$d$, edges labeled by $c$ with edges labeled by $b$ and edges
labeled by $d$ with edges labeled by $c$;

\item it  inserts between  two vertices $v$ and $w$ connected
by an edge of label $a$ two additional vertices $v_1, v_2$ as well
as the following edges:   edges with label $a$ from  $v$ to $v_1$
and from $w$ to $v_2$,   edges with label $b$ and $c$ respectively
between $v_1$ and $v_2$,  edges with label $d$ from $v_1$ to itself
and from $v_2$ to itself.
\end{itemize}

By the very definition of this substitution and the preceding
discussion  we have the following result for the finite Schreier
graphs.


\begin{Proposition}\label{Proposition-Gamma-n}
For every $n\in\NN$, we have $\Theta((\varGamma_{1^n},1^n)) =
(\varGamma_{1^{n+1}},1^{n+1})$.
\end{Proposition}

It is not hard to see that $\Theta$ is compatible with graph
isomorphisms. Thus, $\Theta$ induces a map on the set $\mathcal G_*
(\mathcal{B})$.  We will denote this map also by $\Theta$.

\subsection{The mapping $\Graph$ from words  to  graphs}
Here we will be interested in a special subset  of
$\word(\mathcal{A})$ introduced at the beginning of this section.
This subset $\word^{'} (\mathcal{A})$ consists of all $\omega \in
\word(\mathcal{A})$ satisfying the following two properties:

\begin{itemize}

\item $1$ belongs to the support of $\omega$ (i.e. $\omega (1) \in
\mathcal{A}$).

\item Whenever $\omega_n \omega_{n+1}$ is not equal to $\star \; \star $ then,
exactly one of the two letters $\omega_n$ and $\omega_{n+1}$ is
equal to $a$.
\end{itemize}

Note that the first condition is automatically satisfied for all
elements of $\word(\mathcal{A})$ coming from $\mathcal{A}^\ast$,
$\mathcal{A}^\NN$ and $\mathcal{A}^\ZZ$. The second condition means
that every other letter  is an $a$. Moreover, it ensures that the words in question
start and finish with the letter  $a$.

We will now construct a map from $\word^{'} (\mathcal{A})$ to the
topological  space  $\mathcal G_* (\mathcal{B})$. To picture this
map, it will be useful to think of the letters $x,y,z$ as encoding
the pairs
$$\left(\begin{matrix} b \\ c \end{matrix}\right),  \ \  \left(\begin{matrix} b \\ d
\end{matrix}\right),  \  \left(\begin{matrix} c \\ d \end{matrix}\right), $$
respectively. Roughly speaking the map will replace a letter (at
position $n$)  in the word by  graphs with two vertices ($n$ and
$n+1$) connected by labeled edges according to the specific letter.
In particular, in the case of finite words, the number of vertices
of the graphs will exceed the number of letters of the word by one.
Here are the details.

To  $\omega \in\word^{'} (\mathcal{A})$ we associate
a labeled rooted  graph $\graph(\omega)$ in the following way.

\textit{Vertices.} The
set of vertices is a subset of $\ZZ$  given by the support
$\mbox{support}(\omega)$ of $\omega$ together with $m+1\in \ZZ$ if $\mbox{support}(\omega)$
possesses a maximal element $m$.

\textit{Root.} By the definition of $\word^{'}(\mathcal{A})$
the number $1$ is always a vertex and this vertex is chosen as the
root.

\textit{Edges.} There are edges
between vertices $n,k$ if an only if $|n-k|\leq 1$. Specifically, edges
are assigned between $n$ and $n+1$ and
from $n$ to itself and from $n+1$ to  itself  in the following way:
\begin{itemize}
\item If $\omega (n) = a$, then there is an edge between $n$ and
$n+1$ labeled by $a$.

\item If $\omega (n) = x$, then there are two edges between $n$ and
$n+1$; one is labeled by  $b$ and the other is labeled by
$c$. Moreover, there is an additional edge from $n$ to itself
labeled by $d$ and an additional edge from $n+1$ to itself labeled by  $d$.

\item If $\omega (n) =y$, then there are two edges between $n$ and
$n+1$; one is labeled by $b$ and the other is labeled by
$d$. Moreover, there is an additional edge from $n$ to itself
labeled with $c$ and an additional edge from $n+1$ to itself labeled with  $c$.

\item If $\omega (n) =z$, then there are two edges between $n$ and
$n+1$; one is labeled by $c$ and the other is labeled  by
$d$. Moreover, there is an additional edge from $n$ to itself
labeled with $b$ and an additional edge from $n+1$ to itself labeled with  $b$.

\item If $n$ is the minimal
element of the support of $\omega$ then there are additional three edges labeled with
$b,c,d$ from $n$ to itself. If $n$ is the maximal element of the support of $\omega$
then there are additional three edges labeled with $b,c,d$ from $n$
to itself.
\end{itemize}

The map $\graph$ gives rise to a map $\Graph$ from words to $\mathcal
G_* (\mathcal{B}) $ by taking isomorphism classes via
$$\Graph : \word^{'}(\mathcal{A}) \longrightarrow \mathcal G_* (\mathcal{B})  , \omega \mapsto [\graph (\omega)],$$
where $[ \cdot]$ denotes the isomorphism class.

\begin{Proposition}\label{Graph-continuous} The map $\Graph$ is continuous.
\end{Proposition}
\begin{proof} Obviously, only local information enters the definition of
$\graph$, i.e., two elements of  $\word^{'} (\mathcal{A})$ which
agree on a large interval $J\subset \ZZ$ around $1$ will give rise
to graphs which agree on a large neighborhood  around the root. Now,
agreement on large balls around the root is exactly how the topology
on $\mathcal G_* (\mathcal{B})$ is defined and the continuity
statement easily follows.
\end{proof}

\subsection{The connection: $(X,G)$ as a factor of $(\varOmega_\tau,G)$}
In this section we are going to  connect $\tau$, $\Theta$ and
$\Graph$. More specifically, we will show that $\varOmega_\tau$
admits a natural  action of $G$ by homeomorphisms which is orbit
equivalent to the shift, and that the dynamical system $(X,G)$
defined in Section \ref{Section-Schreier}  is a factor of
$(\varOmega_\tau,G)$, i.e. there exists a continuous surjective map,
the \textit{factor map},   $\psi: \varOmega_\tau\longrightarrow X$
which intertwines the respective actions of $G$. This factor map is
nothing but the restriction of $\Graph$ to $\varOmega_\tau$.

The  following immediate consequence of the definitions shows that
the map $\Graph$ intertwines the actions of the substitutions $\tau$
and $\Theta$.

\begin{Proposition}\label{Intertwining} The substitution $\tau$ maps
$\mathcal{A}^\ast \cap\word^{'}(\mathcal{A})$ into itself and  the equality $\Graph
\circ \tau = \Theta \circ \Graph$ holds on
$\mathcal{A}^\ast \cap\word^{'}(\mathcal{A})$.
\end{Proposition}

We can now state and prove the main lemma
connecting  the substitution $\tau$ and the (finite) Schreier graphs
of Grigorchuk's group $G$. Recall that  $\mathcal F$ maps a finite
or infinite word $w$ to the isomorphism class of the rooted graph
$(\varGamma_w,w)$ (compare Section \ref{Section-Schreier}).

\begin{Lemma}[Connecting $\varGamma$ and $\tau$]
\label{finite-connection}  For all $n\in \NN$ we have
$$\Graph
(\tau^{n-1} (a) ) = [(\varGamma_{1^n}, 1^n)] = \mathcal F ( 1^n)
\;\:\mbox{and}\;\: \Graph (\eta) =[( \varGamma_{1^\infty},1^\infty)]
= \mathcal{F} (1^\infty) .$$
\end{Lemma}
\begin{proof}
The first pair of  relations  follows by an easy induction (compare
Figure \ref{picture-substitution}):

\textit{$n=1$:} We have $\tau^{n-1} (a) = a$. This
translates into a graph $\graph(a)$  with two vertices $1$ and $2$
and one edge between them with label $a$ as well as loops on both
vertices with labels $b,c,d$. This graph is clearly isomorphic to
$\varGamma_1$.

\textit{$n\Longrightarrow n+1$:} We can calculate
\begin{eqnarray*}
\Graph(\tau^{n+1} (a)) &=& \Graph (\tau (\tau^n (a) ) )\\
(\mbox{Prop. \ref{Intertwining}}) \;\: &=& \Theta \circ  \Graph
(\tau^n (a) ) \\
(\mbox{statement for $n$})\;\:\; &=& \Theta (\mathcal F ({1^{n+1}} ) )\\
(\mbox{Prop. \ref{Proposition-Gamma-n} }  ) \:\; &=&
\mathcal F ({1^{n+2}}).
\end{eqnarray*}
This shows the first pair of relations. The second pair
of relations can then be deduced from the continuity of
$\Graph$ shown in Lemma \ref{Graph-continuous} and the continuity of
$\mathcal F$ shown in Lemma \ref{F-continuous}.
\end{proof}

Define the maps $A,B,C,D$ from $\varOmega_\tau$ into itself by
\begin{itemize}
\item  $A (\omega) = ...\omega_0 \omega_1 | \omega_2 ...$ if $\omega_1 =a$ and $ A (\omega) = ....\omega_{-1}| \omega_0 \omega_1...$ if $\omega_0 =a$.
\item $ B (\omega) = ...\omega_0 \omega_1 |\omega_2...$ if $\omega_1 \in\{x,y\}$, $B  (\omega) = ....\omega_{-1} |\omega_0 \omega_1...$ if $\omega_0 \in \{x,y\}$ and $B (\omega) = \omega$ in all other cases.
\item $C(\omega)  =... \omega_0 \omega_1 |\omega_2...$ if $\omega_1 \in\{x,z\}$, $C (\omega) = ....\omega_{-1} |\omega_0 \omega_1...$ if $\omega_0 \in \{x,z\}$ and $C(\omega) = \omega$ in all other cases.
\item $D  (\omega) = ...\omega_0 \omega_1 |\omega_2...$ if $\omega_1 \in\{y,z\}$, $D (\omega) = ....\omega_{-1} |\omega_0 \omega_1...$ if $\omega_0 \in \{y,z\}$ and $D (\omega) = \omega$ in all other cases.
\end{itemize}
Clearly, $A,B,C,D$ are homeomorphisms and involutions. Denote by $H$ the group generated by
$A,B,C,D$ within  the group of homeomorphisms of $\varOmega_\tau$.

\begin{Theorem}[Factor theorem]
\label{Main-connection} The following statements hold:

(a) The group $G$ is isomorphic to the group $H$ via  $\varrho :
G\longrightarrow H$ with  $\varrho(a) =A$, $\varrho(b) = B$,
$\varrho(c) =C$ and $\varrho (d) =  D$.   In particular, there is a
well defined action $\alpha$ of $G$ on $\varOmega_\tau$ given by
$\alpha_g (\omega) :=\varrho (g) (\omega)$ for $g\in G$ and $\omega
\in \varOmega_\tau$ and via this action we obtain a dynamical system
$(\varOmega_\tau,G)$.

(b) The dynamical system $(X,G)$ is a factor of the dynamical system
$(\varOmega_\tau,G)$   via the  map
$$\psi : \varOmega_\tau \longrightarrow X, \omega \mapsto \Graph
(\omega),$$ which is two-to-one.

(c) For every $\omega \in \varOmega_\tau$ the orbits $\{T^n \omega :
n\in \ZZ\}$ and $\{\alpha_g (\omega) : g\in G\}$ coincide.

(d) The dynamical system $(\varOmega_\tau,G)$ is uniquely ergodic
and  the unique $T$-invariant probability measure on
$\varOmega_\tau$ coincides with the unique  $G$-invariant
probability measure on $\varOmega_\tau$.
\end{Theorem}

\begin{Remark} \label{Connection-Bon}
The proof of the theorem shows that   $G$  embeds  into the
topological full group  $[[T]]$ of $(\Omega_\tau,T)$,  as  the
action of the  generators $A,B,C,D$ on $\Omega_\tau$ can  be
represented locally as the action by $T^{\pm 1}$ and $T^0 = id$. In
this context we also mention a   recent  article  of Matte Bon
\cite{Bo}  showing that the group $G$  (and  other groups of
intermediate  growth introduced by the first  author  in
\cite{Gri80})  embed into the topological full group $[[\phi]]$ of a
minimal   subshift $\phi$ over a finite alphabet. While his approach
is different from ours it leads to the same subshift for the group
$G$.
\end{Remark}

\begin{proof} We will show (a) and (b) together.
Let us define
$$\varOmega_1:=\{ T^k \tau^n (a) : n\in \NN\cup \{0\},
0\leq k \leq 2^{n} -1\}\;\:\mbox{and}\;\:
\varOmega_2:= \{T^n \eta : n\in \NN\cup \{0\}\}.$$

Recall also the reflection $\widetilde{}$ on the set
$\varOmega_\tau$ from  Section \ref{section-reflection}, which maps
$\omega = ... \omega_{-1} \omega_0 |\omega_1 \omega_2...$ to
$\widetilde{\omega} = ... \omega_2 \omega_1 | \omega_0 \omega_{-1}
..$.
In the proof, we will use the  reflected version of the fixed point $\eta$ of the substitution $\tau$ given by   $\widetilde{\eta} : \{....,
-2,-1,0,1\}\longrightarrow \mathcal{A}, \widetilde{\eta} (n) = \eta
(2-n)$. Note that
$\widetilde{\eta}$ is a reflected version of $\eta$ which is
additionally shifted so that  $1$ belongs to its support. This
latter  property is needed as we can only associate rooted graphs to
words having $1$ in their support.

\textit{Claim 1.} The closures $ \overline{\varOmega_1}$ and
$\overline{\varOmega_2}$ are compact and the following equalities hold:
$$\overline{\varOmega_1} = \varOmega_\tau \sqcup \{T^n \eta : n\geq 0\}\sqcup
\{T^{-n} \widetilde{\eta} : n \geq 0\} \sqcup \varOmega_1 \:\; \mbox{and}\;\: \overline{\varOmega_2}= \varOmega_\tau \sqcup \varOmega_2.$$

\textit{Proof of the claim.} The sets in question are compact as
they are closed subsets of the compact $\word(\mathcal{A})$. It is
clear that the unions are disjoint. The equalities follow easily
from the minimality of $(\varOmega_\tau,T)$ (as it implies that every
word of $\mbox{Sub}_\tau$ appears in  $p^{(n)}$ for $n$ large enough
 and in  $\eta$).

\textit{Claim 2.} The following equalities hold:
$$\overline{\Graph (\varOmega_1)} = \Graph (\varOmega_\tau) \sqcup \Graph
(\varOmega_1) \sqcup \Graph (\varOmega_2) \;\: \mbox{and}\;\:
\overline{\Graph (\varOmega_2)} = \Graph (\varOmega_\tau) \sqcup
\Graph (\varOmega_2).$$

\textit{Proof of claim.} We only show the statement for
$\varOmega_1$, the statement for $\varOmega_2$ being even easier. As $\Graph$ is continuous and
$\overline{\varOmega_1}$ is compact, we have
$$\overline{\Graph (\varOmega_1)} = \Graph
(\overline{\varOmega_1}).$$ Now, the desired equality follows from
Claim 1 and the fact that the graphs associated to $\eta$ and to
$\widetilde{\eta}$ coincide. Disjointness of the sets in question is
clear and Claim 2 is proven.

Let us now  show that the range of the restriction of $\Graph$
to $\varOmega_\tau$ is $X$ and that the map
$$\psi : \varOmega_\tau \longrightarrow X, \omega \mapsto \Graph
(\omega),$$ is two-to-one.
Recall that $X$ denotes the closure of $\mathcal{F} (\mathcal{T})$ in $\mathcal G_* (\mathcal{B})$  without its isolated points (see Section \ref{Section-Schreier}).
By continuity of the map $\mathcal{F}$ given in Lemma \ref{F-continuous}  we clearly have
$$\mathcal{F} (\partial \mathcal T) \subset\overline{ \{ \mathcal F (x) : x \in \{0,1\}^\ast \} }.$$
Moreover, Lemma \ref{finite-connection}  implies that for
$x\in \{0,1\}^\ast$ with $|x|\geq 1$ there exists a $k\in\{0,\ldots,
2^{|x|-1} -1\}$ with
$$\mathcal F (x)  = \Graph ( T^k \tau^{|x|-1} (a) )$$
(as the graph underlying $\mathcal F (x)$ is exactly $\varGamma_{1^{|x|}}$ and the only choice left is the root).
Putting this together we infer
$$\mathcal{F} (\partial \mathcal T)
 \subset
\overline{\Graph (\varOmega_1)}.$$ By Claim 2, this implies
$\overline{\mathcal{F}  (\partial \mathcal T)} \subset \Graph
(\varOmega_\tau) \sqcup \Graph (\varOmega_1) \sqcup \Graph
(\varOmega_2).$ Clearly, the elements of $\Graph (\varOmega_1)
\sqcup \Graph (\varOmega_2)$ are isolated points and thus
$X\subset\Graph (\varOmega_\tau).$ Conversely, by the previous lemma
we obviously  have $\Graph (\eta) =
\mathcal{F}(1^\infty)\in\mathcal{F}  (\partial \mathcal T).$  The
$G$-invariance of $\partial \mathcal T$ then gives
$$ \Graph (\varOmega_2) = \Graph (\{T^n \eta : n\geq 0\}) \subset
\mathcal F (G \cdot 1^\infty) \subset \mathcal F (\partial \mathcal
T).$$ By  Claim 2 this implies $ \Graph (\varOmega_\tau) \subset
\overline{\Graph(\varOmega_2)} \subset \overline{\mathcal{F}
(\partial \mathcal T)}.$ The points in $\Graph (\varOmega_\tau)$ are
not isolated, as $(\varOmega,T)$ is minimal
This implies $\Graph (\varOmega_\tau) \subset X.$ Put together,
these considerations  give
$$ X = \Graph (\varOmega_\tau).$$

We next  show that the map $\psi$  is two-to-one. 
As we have just shown, for any $x\in X$  there exists an $\omega =
... \omega_{-1} \omega_0 |\omega_1 \omega_2... \in\varOmega_\tau$
with $\Graph (\omega) =x.$ Then,  $\widetilde{\omega}
\in\{a,x,y,z\}^\ZZ$ with
$$\widetilde{\omega} = ... \omega_2 \omega_1 | \omega_0 \omega_{-1} ...$$
belongs to $\varOmega_\tau$  by Proposition \ref{Inversion}, and
clearly satisfies $\Graph (\widetilde{\omega}) = x$ as well, by the
definition of $\Graph$. As there is exactly one $a$ among $\omega_0
\omega_1$, the two sequences $\omega$ and $\widetilde{\omega}$ are
different. This shows that any $x\in X$ has at least two inverse
images under $\Graph$.  Conversely, the Schreier graph $\Graph
(\omega)$  clearly determines the sequence $\omega$ up to one
overall reflection given by $\widetilde{}$ and the statement
follows.

We now show that the group $H$ generated by $A,B,C,D$ is isomorphic to $G$.
Consider the surjective map
$$\delta : \{a,b,c,d\}^\ast \longrightarrow H$$
defined by  $\delta_a := A, \delta_b := B, \delta_c := C, \delta_d
:=D$ and $\delta_v := \delta_{v_1}\circ ...\circ \delta_{v_n}$ for
$v_1\ldots v_n \in \{a,b,c,d\}^n$ with $n\in \NN$. We will show that
$\delta$ induces a group isomorphism from $G$ to $H$.

By definition of $A,B,C,D$,for every
$s\in\{a,b,c,d\}$ the equality
$$ (*)\;\:\; \Graph (\delta_s \omega) = s \;  \Graph (\omega).$$
holds. This implies
$$\Graph  (\delta_v\omega) = v_1\cdot ...\cdot  v_n \;  \Graph (\omega)$$
whenever  $v_1\ldots v_n \in \{a,b,c,d\}^n$ for some $n\in
\NN\cup\{0\}$.

\textit{Claim 3.} For any $\omega \in \varOmega_\tau$ with $T^p
\omega \neq \omega^{(s)}$ for all $s\in\{x,y,z\}$ and $p\in \ZZ$ we
have
$$\Graph^{-1} ( v \; \Graph (\omega)) \cap \{T^k \omega : k\in \ZZ\} =\{\delta_v \omega\}$$
for all $v = v_1\ldots v_n \in \{a,b,c,d\}^n$.

\textit{Proof of the claim.} We already know that the two inverse
images of $\Graph$ of an element  of $X$ differ by a reflection
$\widetilde{}$. The claim then follows  from Corollary
\ref{Orbital-uniqueness}.

\textit{Claim 4.} For any  $v = v_1\ldots v_n \in \{a,b,c,d\}^n$,
 $\delta_v = id$ if and only if $v_1\cdot \ldots \cdot v_n = e\in G$.

\textit{Proof of the claim.} $\Longrightarrow$: We have $\delta_v =
id$. By $\Graph (\delta_v \omega) = v \; \Graph (\omega)$  we can
write
$$ \Graph(\omega) = \Graph (\delta_v \omega) = v \;  \Graph (\omega)$$
for any $\omega\in\varOmega_\tau$. This shows that $v$ acts as the
identity on $X$. Invoking  the factor map $\phi : X\longrightarrow
\partial \mathcal T$ we infer that the  action of $v$ on $\partial \mathcal
T$ is the identity as well. As  $G$ acts faithfully on $\partial
\mathcal T$, we conclude   $v = e\in G$.

$\Longleftarrow$: Assume $v =v_1\cdot \ldots \cdot v_n = e\in G$.
By Claim 3 we have
$$ \{\delta_v \omega\} = \Graph^{-1} ( v \Graph (\omega)) \cap \{ T^n \omega : n\in \ZZ\} = \{\omega\} $$
for all $\omega\in \varOmega_\tau$ with $T^p \omega \neq
\omega^{(s)}$ for all $s\in\{x,y,z\}$ and $p\in \ZZ$. This shows
$\delta_v = id$ on a dense set in $\varOmega_\tau$ and $\delta_v =
id$ follows.

From Claim 4 and the definition of $\delta$ we directly obtain that
there is a group isomorphism $\varrho : G\longrightarrow H$ with
$\varrho(a) =A$, $\varrho(b) = B$, $\varrho(c) =C$ and $\varrho (d)
=  D$.

This completes the proof of statements (a) and (b).
The statement (c) is a direct consequence of the definition of
$A,B,C,D$.
Finally, the statement (d) follows easily by considering cylinder sets around the origin of the form
$$\{\omega \in
\varOmega :  \omega (m)\ldots \omega (m+ |v| -1) = v\}$$ for $v\in
\mathcal W_\tau$ and $m\in \ZZ$ with $m< 0$ and $m + |v|-1 >0$, and
noting that the generators of $G$ act on such sets either as
identity or as $T$ or as $T^{-1}$.
\end{proof}

\section{Application to spectral theory of the Laplacians associated to the Schreier graphs}\label{Main-results}
In this section we will bring the results of all the previous
sections together in order to treat the spectral theory of the
operators $M_x (t,u,v,w)$, $x\in X$. We will  use Theorem
\ref{Main-connection} to show that each of these  operators  is
unitarily equivalent to  an  operator from a  family  $H_\omega$,
${\omega\in\varOmega_\tau}$, of Schroedinger operator with aperiodic
order. Spectral properties of Schroedinger operators with aperiodic
order in turn were  heavily studied in the past twenty five years or
so (see the introduction for further details and references). We
will then use the combinatorial results  of Section
\ref{The-substitution} and the known theory  of such Schroedinger
operators to describe  spectral properties of the $H_\omega$,
$\omega\in\varOmega_\tau$  and hence the $M_x (t,u,v,w)$, $x\in X$.



Let four parameters  $t,u,v,w\in \RR$ be given. Set $D := u + v + w$
and define $ f: \varOmega_\tau\longrightarrow \RR$ and $g
:\varOmega_\tau \longrightarrow \RR$ by
$$ f(\omega) := \left\{
\begin{array}{ccc} t &:&
\omega_0 = a \\   D - w &:& \omega_0 = x\\ D - v &:& \omega_0 = y\\
D- u & : &  \omega_0 = z\end{array} \right.\;\:\;\mbox{and}\;\:
g(\omega) := \left\{
\begin{array}{c}  w : \omega_{-1} \omega_0 \in\{ ax, xa\} \\ v :
\omega_{-1}  \omega_0 \in \{ay,ya\}\\
u : \omega_{-1} \omega_0 \in  \{az,za\}
\end{array} \right..$$
For us the actual definition of these functions is not that
important. What matters is that theses functions are \textit{locally
constant} (i.e. their values at $\omega$ depond only on $\omega
(-N)\ldots \omega (N)$ for some fixed natural number $N$). Moreover,
we will use  that these functions are not periodic if $u = v = w$
does not hold (as follows directly from  Proposition
\ref{Proposition-Shift-non-periodic}). We will need the set
$$\mathcal{P}:=\{(t,u,v,w) \in \RR^4 : t\neq 0, u+ v \neq 0, u + w
\neq 0, v + w \neq 0\}$$ as $f$ does not vanish anywhere if and only
if the parameters $(t,u,v,w)$ belong to $\mathcal{P}$.

Define for each $\omega\in\varOmega_\tau$ the selfadjoint bounded
operator $H_\omega :\ell^2 (\ZZ)\longrightarrow \ell^2 (\ZZ)$ via
$$(H_\omega \varphi) (n) = f(T^{n-1}\omega) \varphi (n-1) + f (T^n
\omega) \varphi (n+1) + g(T^n \omega) \varphi (n)$$ for all $\varphi
\in \ell^2 (\ZZ)$.

\begin{Proposition}\label{unitary-equivalence}
Let $(t,u,v,w) \in \RR^4$ be given. Consider  $x \in X$. Then, there
exists an  $\omega\in \varOmega_\tau$  such that  $H_\omega$ is
unitarily equivalent to $M_x (t,u,v,w)$.
\end{Proposition}
\begin{proof}  By the surjectivity of $\psi$ in Theorem \ref{Main-connection},
there exists an $\omega\in\varOmega_\tau$ with $\psi (\omega) = x$.
By the definition of $\psi$, the graph $\graph (\omega)$ is then a
representative of $x$ (recall that $x$ is an isomorphism class of
graphs). Denote its set of vertices by $V$. By the definition of
$\graph$ there exists a map $j:  V \longrightarrow \ZZ$ such that
the edges emanating from $p\in V $ are described by
$\omega_{j(p)-1}\omega_{j(p)}$. The map $j$ gives rise to  the
unitary map
$$U : \ell^2 (\ZZ)\longrightarrow \ell^2 (V), \varphi \mapsto \varphi
\circ j.$$  Via this unitary map the operator $M_x (t,u,v,w)$ is
converted to the operator $$\widetilde{H}_x :=U^{-1} M_x (t,u,v,w) U
$$ on $\ell^2 (\ZZ).$
 For $\varphi \in \ell^2 (V)$,  the value of
$M_x (t,u,v,w) \varphi (p)$ at a $p\in V$ is determined by the
(weights of the) edges emanating from $p$ (and by $\varphi$).  As
these edges are encoded by $\omega_{j(p)-1} \omega_{j(p)}$ the value
of $(\widetilde{H}_x U^{-1} \varphi) (j(p))$ is then encoded by the
values of $\omega_{j(p)-1} \omega_{j(p)}$ as well. The exact
dependance on $\omega_{j(p) -1} \omega_{j(p)}$  follows from an
direct calculation and yields $\widetilde{H}_x = H_\omega$.
\end{proof}

Combining the previous proposition  with known  results on Cantor
spectrum of Lebesgue measure zero, we obtain the following result.

\begin{Theorem}[Cantor spectrum] \label{Main-Laplacian-Cantor}
 Let  $(t,u,v,w)\in\mathcal{P}$ such that  $u = v = w$ does not hold. Then,
  there exists a Cantor set  $\Sigma$ of
Lebesgue measure zero with $\Sigma = \sigma (M_x)$ for all $x\in X$.
\end{Theorem}
\begin{proof} By Corollary 6.6 of  \cite{BP} (see Corollary 2.3 in   \cite{Lenz}
for the case $f\equiv 1$ as well)   a family of  Schroedinger
operators associated to a linearly repetitive  subshift has Cantor
spectrum of Lebesgue measure zero   provided the underlying
functions $f$ and $g$ are locally constant and not periodic and $f$
does not vanish. (Corollary 6.6 even allows for a weakening of
linear repetitivity.) Now, the subshift $(\varOmega_\tau,T)$ is
linearly repetitive by Theorem \ref{Theorem-tau-linear-repetitive}
and clearly $f$ does not vanish if the parameters belong to the set
$\mathcal{P}$. Moreover, as discussed already at the beginning of
this section  $f$ and $g$ are locally constant and
 they are not periodic if $u = v = w$ does not hold. So, from
\cite{BP} we obtain that there exists a Cantor set $\Sigma$ with
$\sigma (H_\omega) = \Sigma$ for all $\omega \in \varOmega_\tau$.
From Proposition \ref{unitary-equivalence}, we then obtain the
desired result.
\end{proof}

\begin{Remark}  The case $u = v = w$ is  treated in \cite{BG} and an
explicit description of the spectrum (in terms of the values of $u$
and $t$) can be found there.  From \cite{BG} it is also already
known that the spectrum is independent of $x\in X$ (where  arbitrary
values of $t,u,v,w$ are fixed).

\end{Remark}


We can also use the so-called Gordon Lemma  (after \cite{Gor}) from
the theory of Schroedinger operators with aperiodic order  to
exclude eigenvalues. The Gordon Lemma does not seem to be in the
literature in the generality needed here. So we  briefly discuss
some  details next.

\begin{Lemma}[Gordon Lemma]\label{l:Gordon} Let $(f_n)$  and
$(g_n)$ be two bounded sequences in $\RR$ with $f_n \neq 0$ for all
$n\in\ZZ$. Consider the self-adjoint operator $H$ acting on $\ell^2
(\ZZ)$ via
$$(H \varphi) (n) = f_n \varphi (n+1) + f_{n-1} \varphi (n-1) + g_n
\varphi (n).$$ If there exists a sequence of natural numbers $L_k\to
\infty$ such that for every $k$
$$(f_n,g_n) = (f_{n-L_k}, g_{n-L_k}) = (f_{n + L_k}, g_{n +
L_k})$$ with $n=1,\ldots, L_k$, then the operator $H$ does not have
any eigenvalues.
\end{Lemma}
\begin{proof}
In the case $f\equiv 1$ this is well-known, see e.g. the review
\cite{Dam}. The proof presented there can be adapted to our more
general case:  Choose $E\in\RR$ arbitrarily and consider  functions
$u : \ZZ\longrightarrow \CC$ with $ f_n u (n+1) + f_{n-1} u (n-1) +
(g_n - E ) u (n) =0 $ for all $n\in\ZZ$. Define for such a function
and each $n\in \ZZ$
$$\widetilde{u}_n :=\left(\begin{matrix} u(n+1) \\ f_n u(n)
\end{matrix}\right)\in\CC^2.$$ Let  the \textit{transfer matrix} $\tilde{M}_n$ be the
linear map which sends $\widetilde{u}_0 $ to $\widetilde{u}_n$ (for
any such $u$).

We have to show that  such a function  $u$ does not belong to
$\ell^2 (\ZZ)$ except if $u$ vanishes everywhere.

Note that   all $\tilde{M}_n$ have determinant equal to one (see
e.g. \cite{BP}). Given this, we can now continue as in the proof of
the usual Gordon lemma to conclude (for all $k\in\ZZ$)
$$\max\{ \|\widetilde{u}_{2 L_k}\|, \|\widetilde{u}_{L_k}\|,
\|\widetilde{u}_{-L_k}\|\} \geq \frac{1}{4} \|\widetilde{u}_0\|,$$
where $\|\cdot\|$ denotes the Euclidean norm on $\CC^2$. As $(f_n)$
is bounded and non-zero  and $L_k\rightarrow\infty$, the function
$u$ can not belong to $\ell^2$ (except if it vanishes everywhere).
\end{proof}

\begin{Corollary}\label{Absence-eigenvalues-abstract} Let $(\varOmega,T)$ be a subshift over a finite alphabet, $\omega\in\varOmega$.
Let $f,g :\varOmega\longrightarrow \RR$ be locally constant  such
that $f$ is nowhere zero. 
Assume that there exists a sequence of finite words $w_n$ with
$|w_n|\to\infty$ as well as non-empty prefixes $v_n$ of $w_n$ with
$|v_n|\to\infty$ such that
$$ \omega = ... w_n w_n  | w_n v_n...,$$
where $|$ denotes the position of the origin. Then, $H_\omega$ does
not have any eigenvalues.
\end{Corollary}
\begin{proof} 
By assumption, $f,g$ are locally constant. After applying a suitable
shift we can then assume without loss of generality that there
exists an $N\in\NN$ such that the values of $f$ and $g$ only depend
on the positions $0,\ldots, N$. For sufficiently large values of $n$
we will have $|v_n|\geq N$. Thus, the sequences $k\mapsto
f(T^k\omega)$ and $k\mapsto  g(T^k \omega)$ will satisfy the
assumptions of the previous lemma. The lemma then gives the desired
statement.
\end{proof}

\begin{Theorem}[Absence of eigenvalues] \label{Main-Laplacian-Absence-eigenvalues} Assume $(t,u,v,w)\in\mathcal{P}$ and $u = v = w$ does not hold.

(a) For almost every $x\in X$ the operator $M_x (t,u,v,w)$ does not
have eigenvalues. In particular, for $\mu$-almost every $\xi \in
\partial \mathcal T$ the operator $M_\xi$ does not have eigenvalues.

(b) For every $x\in X_1$ the operator $M_x$ does not have
eigenvalues.
\end{Theorem}
\begin{proof}
The first statement of (a) with $H_\omega$ instead of $M_x
(t,u,v,w)$ can be shown by combining Corollary
\ref{Absence-eigenvalues-abstract} with Lemma
\ref{Three-blocks-in-omega-s}. The actual statement for the $M_x
(t,u,v,w)$ then follows from  Proposition \ref{unitary-equivalence}.
The last statement of (a) then follows from the discussion in
Section \ref{Section-Schreier}, which provided an almost everywhere
one-to-one map  $\phi : X\longrightarrow \partial \mathcal{T}$.

Part  (b) follows by combining Corollary
\ref{Absence-eigenvalues-abstract} with Lemma
\ref{Three-blocks-in-almost-every-omega} and Proposition
\ref{unitary-equivalence}.
\end{proof}

\begin{Remark}
 In this section we were concerned with $(t,u,v,w)\in \mathcal{P}$.
For  every $(t,u,v,w)\notin\mathcal{P}$ the operator in question can
be decomposed as a sum of finitely many
 finite dimensional operators each appearing with infinite
multiplicity. Thus, the  spectrum is pure point with finitely many
eigenvalues each with infinite multiplicity.
\end{Remark}

\section{Outlook}\label{Outlook}
Besides generalization of our results and further study of groups
with linear Schreier graphs, as described in the introduction, we
would like to discuss here various further alleys of research. We
plan to work on these issues in the future.

\textit{Absence of eigenvalues.} Our results show absence of
eigenvalues for 'most' of the operators $M_x (t,u,v,w)$, $x\in X$,
as well as for a few particularly interesting special cases.  It is
an open question whether this absence of eigenvalues actually holds
for all values of the parameters $x\in X$.

\textit{Other self-similar groups.} The group $G$ studied in this
paper belongs to the class of self-similar groups. In many cases,
self-similarity of a group action on a regular rooted tree leads to
a finite collection  of rules that allow one to construct
inductively the Schreier graphs $\{\Gamma_n\}_{n\geq 1}$ for the
action on the levels of the tree. Linearity of Schreier graphs was important in our approach, however, it is possible that
similar considerations can be carried out also for other
self-similar groups with Schreier graphs of higher complexity . This is especially interesting in relation with
the  spectral problem of Schreier graphs of other 'important'
self-similar groups like the Basilica group or the Hanoi Tower group
where only partial results exist so far.

\textit{Finer spectral properties.}  Various tools are available
nowadays to study finer Hausdorff properties of spectra of
Schroedinger operators associated with aperiodic order (see e.g.
\cite{DEG} for a recent survey). This machinery is (mostly) based on
showing that the spectrum is dynamically generated by some maps.
These maps arise via traces of periodic approximants and are called
'trace maps'. It is tempting to think that these ideas can be
applied to our example. In fact, this may tie in well with the way
how the infinite Schreier graphs are generated via approximation by
$\varGamma_n$'s.

\textit{Exploration of spectra via the map $F$.} Our approach to the
spectral properties of the Laplacian is very different from the
one in \cite{BG}. Their approach works via a map $F$
encoding the self-similar structure of the spectrum, and the
spectrum is then seen as a
the  intersection  of a line  and a certain $F$-invariant set
$\Delta$  in Euclidean space that has to be identified among all
$F$-invariant subsets. A similar method was also applied in
\cite{DDN} for spectral computations of an oriented Laplacian
related to the dimers model on the Schreier graphs of self-similar
groups.  It will be very interesting to see whether the knowledge of
spectral properties gained in the present paper can help to further
explore the approach to the Laplacian spectrum via the map $F$.
Also, it seems not unreasonable that  there is a connection between
the map $F$ and the approach to spectral properties via trace maps
mentioned in the previous point.

\end{document}